\documentclass{amsart}

\usepackage{amssymb} %needed for \mathbb
\usepackage{amsthm} %needed for "proof" environment
\usepackage{amsmath} %needed for environments align, ...
\usepackage[all]{xy} %needed for (not too fancy) commutative diagram
\usepackage{paralist}

\begin{document}
\let\oldref\ref
\renewcommand{\ref}[1]{(\oldref{#1})}
\newcommand{\mathref}[1]{$($\oldref{#1}$)$}

\newcommand{\scrC}{\ensuremath{\mathcal{C}}}
\newcommand{\scrF}{\ensuremath{\mathcal{F}}}
\newcommand{\scrG}{\ensuremath{\mathcal{G}}}
\newcommand{\scrI}{\ensuremath{\mathcal{I}}}
\newcommand{\scrL}{\ensuremath{\mathcal{L}}}
\newcommand{\scrM}{\ensuremath{\mathcal{M}}}
\newcommand{\scrP}{\ensuremath{\mathcal{P}}}
\newcommand{\scrQ}{\ensuremath{\mathcal{Q}}}

\newcommand{\psif}{\ensuremath{\psi_\scrF}}

\newcommand{\cc}{\ensuremath{\mathbb{C}}}
\newcommand{\kk}{\ensuremath{\mathbb{K}}}
\newcommand{\nn}{\ensuremath{\mathbb{N}}}
\newcommand{\pp}{\ensuremath{\mathbb{P}}}
\newcommand{\qq}{\ensuremath{\mathbb{Q}}}
\newcommand{\rr}{\ensuremath{\mathbb{R}}}
\newcommand{\zz}{\ensuremath{\mathbb{Z}}}

\newcommand{\affine[2]}{\ensuremath{\mathbb{A}^{#1}(#2)}}
\newcommand{\aNk}{\affine[N]{\kk}}
\newcommand{\ank}{\affine[n]{\kk}}

\newcommand{\ccc}{\ensuremath{\mathfrak{c}}}
\newcommand{\FFF}{\ensuremath{\mathfrak{F}}}
\newcommand{\GGG}{\ensuremath{\mathfrak{G}}}
\newcommand{\mmm}{\ensuremath{\mathfrak{m}}}
\newcommand{\hhh}{\ensuremath{\mathfrak{h}}}
\newcommand{\iii}{\ensuremath{\mathfrak{I}}}
\newcommand{\ppp}{\ensuremath{\mathfrak{p}}}

\newcommand{\torus}{\ensuremath{\mathbb{T}}}
\newcommand{\sheaf}{\ensuremath{\mathcal{O}}}

\newcommand{\im}{\ensuremath{\Rightarrow}}
\newcommand{\bs}{\backslash}
\newcommand{\dsum}{\ensuremath{\bigoplus}}
\newcommand{\into}{\ensuremath{\hookrightarrow}}
\newcommand{\onto}{\twoheadrightarrow}
\newcommand{\tensor}{\otimes}

\newcommand{\D}[2]{\frac{\partial #2}{\partial #1}}
\newcommand{\Dd}[2]{\frac{\partial^2 #2}{\partial #1^2}}
\newcommand{\DD}[3]{\frac{\partial^2 #3}{\partial #1 \partial #2}}

\newcommand{\finv}{\ensuremath{f^{-1}}}

\newtheorem{thm}{Theorem}[subsection]
\newtheorem*{thm*}{Theorem}
\newtheorem{lemma}[thm]{Lemma}
\newtheorem*{lemma*}{Lemma}

\newtheorem{prolemma}{Lemma}[thm]
\newtheorem{prop}[thm]{Proposition}
\newtheorem{cor}[thm]{Corollary}
\newtheorem{claim}{Claim}[thm]
\newtheorem*{claim*}{Claim}
\newtheorem*{conjecture*}{Conjecture}

\theoremstyle{definition} 
\newtheorem{example}[thm]{Example}
\newtheorem*{example*}{Example}
\newtheorem{thexample}{Example}[thm]
\newtheorem*{defn*}{Definition}
\newtheorem*{fact*}{Fact}
\newtheorem{rem}[thm]{Remark}

\theoremstyle{remark}
\newtheorem*{rem*}{Remark}
\newtheorem*{notation*}{Notation}
\newtheorem*{question*}{Question}

\newcounter{Example}

\theoremstyle{plain}
\newtheorem{THM}{Theorem}
\newtheorem{PROP}{Proposition}
\newtheorem{LEMMA}{Lemma}

\newcommand{\cl}{\text{Cl}\:}
\newcommand{\conv}{\text{conv}\:}
\newcommand{\gr}{\text{gr}\:}
\newcommand{\proj}{\text{Proj}\:}
\newcommand{\spec}{\text{Spec}\:}
\newcommand{\supp}{\text{Supp}\:}

\newcommand{\identity}{\mathbb{I}}
\newcommand{\ntorus}{(\cc^*)^n}
\newcommand{\WP}{\ensuremath{\bf{W}\pp}}

\newcommand{\cX}{\ensuremath{\cc[X_1,\ldots,X_n]}}
\newcommand{\cx}{\ensuremath{\cc[x_1,\ldots,x_n]}}
\newcommand{\cxstar}{\ensuremath{\cc[x_1,x_1^{-1},\ldots,x_n,x_n^{-1}]}}

\newcommand{\p}{\ensuremath{\mathcal{P}}}
\newcommand{\pnu}{\ensuremath{\p^\nu}}
\newcommand{\pd}{\ensuremath{\p_d}}
\newcommand{\pdminus}{\ensuremath{\p_{d^{-}}}}
\newcommand{\pnubar}{\ensuremath{\overline{\p}^\nu}}

\newcommand{\cXDelta}{\ensuremath{\cX[X_0]_\Delta}} %most probably needed only in the intro

\newcommand{\Md}{Pseudo-degree\ }
\newcommand{\md}{pseudo-degree\ }%md == milman degree
\newcommand{\gf}{\gff\ } %gf == good filtration
\newcommand{\gff}{semidegree} %gf == good filtration
\newcommand{\Gf}{\Gff\ }
\newcommand{\Gff}{Semidegree}
\newcommand{\sgf}{\sgff\ } %gf == good filtration
\newcommand{\sgff}{quasidegree}
\newcommand{\Sgf}{\Sgff\ }
\newcommand{\Sgff}{Quasidegree}
\newcommand{\inducedfiltrn}{\inducedfiltrnn\ }
\newcommand{\inducedfiltrnn}{induced filtration}
\newcommand{\Inducedfiltrn}{\Inducedfiltrnn\ }
\newcommand{\Inducedfiltrnn}{Induced filtration}
\newcommand{\InducedFiltrn}{\InducedFiltrnn\ }
\newcommand{\InducedFiltrnn}{Induced Filtration}
\newcommand{\genfiltrn}{\genfiltrnn\ }
\newcommand{\genfiltrnn}{filtration}
\newcommand{\Genfiltrn}{\Genfiltrnn\ }
\newcommand{\Genfiltrnn}{Filtration}
\newcommand{\gengf}{\gengff\ }
\newcommand{\gengff}{\gff}
\newcommand{\Gengf}{\Gengff\ }
\newcommand{\Gengff}{\gff}
\newcommand{\gensgf}{\gensgff\ }
\newcommand{\gensgff}{\sgff}
\newcommand{\Gensgf}{\Gensgff\ }
\newcommand{\Gensgff}{\sgff}
\newcommand{\inducedgenfiltrn}{\inducedgenfiltrnn\ }
\newcommand{\inducedgenfiltrnn}{induced \genfiltrnn}
\newcommand{\Inducedgenfiltrn}{\Inducedgenfiltrnn\ }
\newcommand{\Inducedgenfiltrnn}{Induced \genfiltrnn}
\newcommand{\InducedGenFiltrn}{\InducedGenFiltrnn\ }
\newcommand{\InducedGenFiltrnn}{Induced \Genfiltrnn}

\newcommand{\simpleatzero}{\simpleatzeroo\ }
\newcommand{\simpleatzeroo}{compact}
\newcommand{\nonnegative}{\nonnegativee\ }
\newcommand{\nonnegativee}{non-negative}
\newcommand{\complete}{\completee\ }
\newcommand{\completee}{complete}

\newcommand{\thechar}{\nu}
\newcommand{\thering}{\p}
\newcommand{\mdring}{\ensuremath{\thering^\thechar}}
\newcommand{\ringd}{\ensuremath{\thering_d}}
\newcommand{\ringdminus}{\ensuremath{\thering_{d^{-}}}}
\newcommand{\mdringbar}{\ensuremath{\overline{\thering}^\thechar}}

\newcommand{\filtrationchar}{\ensuremath{\mathcal{F}}}
\newcommand{\filtrationring}{\ensuremath{A}}
\newcommand{\profing}{\profingg{\filtrationring}{\filtrationchar}}
\newcommand{\profingg}[2]{\ensuremath{{#1}^{#2}}}
\newcommand{\profinggg}[1]{\profingg{\filtrationring}{#1}}
\newcommand{\profingy}[1]{\profingg{\filtrationring}{\filtrationchar_{#1}}}
\newcommand{\gring}{\gringg{\filtrationring}{\filtrationchar}}
\newcommand{\gringg}[2]{\ensuremath{\gr {#1}^{#2}}}
\newcommand{\gringgg}[1]{\gringg{\filtrationring}{#1}}
\newcommand{\gringy}[1]{\gringg{\filtrationring}{\filtrationchar_{#1}}}

\newcommand{\ld}{\mathfrak{L}}
\newcommand{\xbar}{\bar X}

\title[On projective completions determined by 'degree-like' functions]{On projective completions of affine varieties determined by 'degree-like' functions}
\author{Pinaki Mondal}
\address{Department of Mathematics\\
University of Toronto \\
Toronto, ON}
\email{pinaki@math.toronto.edu}

\thanks{The author expresses his gratitude to his advisor Professor Pierre Milman for posing the questions and guidance throughout his work.}

\subjclass[2000]{14A10}
\keywords{Completion, Projective completion, Degree like function, semidegree, quasidegree, affine Bezout type, Degree of completion}

\begin{abstract}
We study projective completions of affine algebraic varieties which are given by filtrations, or equivalently, `degree like functions' on their rings of regular functions. For a quasifinite polynomial map $P$ (i.e. with all fibers finite) of affine varieties, we prove that there are completions of the source that do not add points at infinity for $P$ (i.e. in the intersection of completions of the hypersurfaces corresponding to a generic fiber and determined by the component functions of $P$). Moreover we show that there are `finite type' completions with the latter property, determined by the maximum of a finite number of `semidegrees', i.e. maps of the ring of regular functions excluding zero, into integers, which send products into sums and sums into maximas (with a possible exception when the summands have the same semidegree). We characterize the latter type completions as the ones for which the ideal of the `hypersurface at infinity' is radical. Moreover, we establish a one-to-one correspondence between the collection of minimal associated primes of the latter ideal and the unique minimal collection of semidegrees needed to define the corresponding degree like function. We also prove an `affine Bezout type' theorem for quasifinite polynomial maps $P$ which admit semidegrees such that corresponding completions do not add points at infinity for $P$.
\end{abstract}

\maketitle

\section*{Introduction} \label{sec-intro}
Let $X$ be an affine algebraic variety of dimension $n$ over an algebraically closed field $\kk$. A projective completion of $X$ is an open immersion given by an algebraic morphism $\psi: X \into Z$ of $X$ onto a dense open subset of an algebraic subvariety $Z$ of some projective space $\pp^N$. The point of departure for the study in this article is the (well known) observation that {\em filtrations} on the coordinate ring $\kk[X]$ of $X$ correspond to projective completions of $X$. Giving a filtration $\scrF$ on $\kk[X]$, on the other hand, is equivalent to defining a {\em degree like function} $\delta: \kk[X] \to \zz$ which satisfies the following properties:
\begin{enumerate}
\item \label{deg1} $\delta(f+g) \leq \max\{\delta(f), \delta(g)\}$ for all $f, g \in \kk[X]$, with $<$ in the preceding equation implying $\delta(f) = \delta(g)$.
\item \label{deg2} $\delta(fg) \leq \delta(f) + \delta(g)$ for all $f, g \in \kk[X]$.
\end{enumerate}
The notion of a degree like function generalizes the usual and weighted degrees on the polynomial ring (which correspond respectively to the usual and weighted projective completions of the affine space). But observe that both the usual and weighted degrees satisfy property \ref{deg2} with exact {\em equality} instead of the inequality, and therefore are examples of a special class of degree like functions which we call {\em semidegrees}.\\

Given a projective completion $Z$ of $X$ determined by a degree like function $\delta$, we define a `normalized' degree like function $\bar \delta$ which is a maximum of finitely many semidegrees. When $X$ is normal, the corresponding completion $\bar Z$ is isomorphic to the normalization of $Z$. The construction of $\bar \delta$ from $\delta$ generalizes the construction of the normal {\em toric completion} $X_\scrP$ of $\ntorus$ determined by a convex integral polytope $\scrP$ from the toric variety determined by an arbitrary finite subset $S$ of integral points in $\scrP$ whose convex hull is $\scrP$. The irreducible components of $X_\scrP\setminus\ntorus$ correspond to facets of $\scrP$. Similarly, we establish a one-to-one correspondence between the components of the $\bar Z \setminus X$ and the unique minimal collection of semidegrees needed to define $\bar\delta$.\\

This work started as a project to understand `affine Bezout type theorems' which provide formulae for the number of solutions of systems of equations on affine varieties, the systems being generic in some suitable sense. Curiously, constructions in \cite{kush-poly-bezout}, \cite{bern}, \cite{khovanskii-genus}, \cite{rojas-convex}, \cite{rojas-toric}, \cite{rojas-wang} have one property in common - they involve completions $\psi: X \to Z$ of the affine variety $X$ which satisfy the following property with respect to some polynomial map $f = (f_1, \ldots, f_n): X \to \cc^n$ with finite fibers: 
\begin{align}
\label{bezout-completion-property}
\parbox[c]{0.85\textwidth}{for generic $a = (a_1, \ldots, a_n) \in \cc^n$, $\bar{H}_1(a) \cap \cdots \cap \bar{H}_n(a) \cap (Z\setminus X) = \emptyset$, where $H_i(a) := \{x \in X: f_i(x) = a_i\}$ and $\bar H_i(a)$ is the closure of $H_i(a)$ in $Z$ for all $1 \leq i \leq n$.}  \tag{$*$}
\end{align}
This observation suggests the following definitions: for finitely many closed subvarieties $V_1, \ldots, V_k$ of $X$, a completion $\psi: X \into Z$ {\em preserves the intersection of $V_1, \ldots, V_k$ at $\infty$} if $\bar{V}_1 \cap \cdots \cap \bar{V}_k \cap X_\infty =  \emptyset$, where $X_\infty := Z\setminus X$ is the set of `points at infinity' and $\bar{V}_j$ is the closure of $V_j$ in $Z$ for every $j$. Given a polynomial map $f=(f_1, \ldots, f_q):X \to \affine[q]{\kk}$, $a = (a_1, \ldots, a_q) \in f(X)$ and a completion $\psi$, we say that $\psi$ {\em preserves the fiber $\finv(a)$ at $\infty$} if $\psi$ preserves the intersection of the hypersurfaces $H_i(a) := \{x \in X: f_i(x) = a_i\}$, $i = 1, \ldots, q$. Finally, when $f$ is quasifinite, i.e. has only finite fibers, we say that completion $\psi$ {\em preserves $f$ at $\infty$} provided it satisfies \ref{bezout-completion-property}, i.e. for all $a$ in a non-empty Zariski open subset of $f(X)$, $\psi$ preserves $\finv(a)$ at $\infty$. We prove in this article that given any polynomial map $f:X \to \ank$ with finite fibers, there exists a degree like function $\delta$ on $\kk[X]$ such that $\delta$ is the maximum of finitely many semidegrees and the corresponding completion of $X$ satisfies \ref{bezout-completion-property}. When $\delta$ is itself a semidegree, we derive an affine Bezout type formula for the size of generic fibers of $f$ in terms of degree $D$ of a $d$-uple completion of the resulting variety. Using \cite{khovanskii-kaveh}, we provide a description of $D$ in terms of the volume of a convex body determined by $\delta$. We also describe an iterative procedure (starting with a weighted degree) of a construction of a semidegree and, for $X= \ank$, provide a simple algebraic formula for $D$.\\

Our article is organized as follows. In section \oldref{subsec-filtrintro-defn} we define the notion of a filtration and give a characterization of the completions coming from filtrations. After several basic examples of filtrations and corresponding completions in section \oldref{subsec-filtrintro-examples}, we give an example of a projective completion of an affine variety $X \neq \ank$ which does not come from a filtration. On the other hand, we ask the following
\begin{question*}
Does $\ank$ admit a projective completion not induced by a filtration?
\end{question*}
Section \oldref{subsec-filtrintro-existence} is where we consider the question: ``given a collection of subvarieties of $X$ with finite intersections, when does a completion of $X$ preserve their intersection at $\infty$?'' We prove in this section that given a map $f: X \to \ank$ with finite fibers, there is a filtration on $\kk[X]$ such that the corresponding completion satisfy \ref{bezout-completion-property}. We introduce degree like functions corresponding to filtrations in section \oldref{subsec-semiquasintro} and present examples of completions determined by semidegrees and {\em quasidegrees}, which are the maxima of finitely many semidegrees. Here we also describe an `iterated' recipe of producing new semidegrees from an old one. In section \oldref{subsec-semiquasidegree-properties} we prove our main results on projective completions determined by quasidegrees. Our first theorem classifies the filtrations determined by semi- and quasidegrees. As a corollary we deduce that for a completion $\psi: X \into Z$ given by a quasidegree $\delta$, the irreducible components of $X_\infty := Z\setminus X$ are in a one-to-one correspondence with the unique minimal collection of semidegrees defining $\delta$. Given an arbitrary completion which comes from a filtration and preserves the intersection at $\infty$ of a collection of subvarieties, we show in our (main) existence theorem that there is a completion determined by a quasidegree which preserves the intersection at $\infty$ of the subvarieties in the collection. We make use of the latter to conclude the existence of completions determined by quasidegrees which preserve a given quasifinite polynomial map $P$ at $\infty$. Not all quasifinite maps admit completions of the latter type determined by a  quasidegree which is also a semidegree, as our example of section \oldref{subsec-semi-ctr-example} shows. \\

In section \oldref{sec-bezout} we prove the Bezout theorem for a semidegree $\delta$ in terms of the degree $D$ of the resulting projective variety in an appropriate ambient space. When $\kk = \cc$, we give a description of $D$ in terms of the volume of a convex body in $\rr^n$ using \cite{khovanskii-kaveh}. For an `iterated' semidgree we present a simple formula for $D$. Finally, in the appendix we provide a proof that Bernstein's formula \cite{bern} for counting zeros of a Laurent polynomial system holds with an equality if and only if the completion $\ntorus \into X_\scrP$ satisfies \ref{bezout-completion-property}, where $\scrP$ is the Minkowski sum of the supporting polytopes of components of $f$. The results of this article have been announced in \cite{announcement}.

\section{Filtrations} \label{sec-filtrintro}
\subsection{Definitions and Preliminaries} \label{subsec-filtrintro-defn}
\newcommand{\af}{\profing}
\newcommand{\xf}{X^\scrF}
Throughout this article $A$ will be a finitely generated algebra over $\kk$.
\begin{defn*} A \genfiltrn $\scrF$ on $A$ is a family $\{F_i:\ i \in \zz\}$ of $\kk$-vector subspaces of $A$ such that
\begin{enumerate}
\item $F_i \subseteq F_{i+1}$ for all $i \in \zz$, \label{genfiltrn-prop-1}
\item $1 \in F_0$, \label{genfiltrn-prop-2}
\item $A = \bigcup_{i \in \zz}F_i$, and \label{genfiltrn-prop-3}
\item $F_iF_j \subseteq F_{i+j}$ for all $i,j \in \zz$. \label{genfiltrn-prop-4}
\end{enumerate}
\end{defn*}

\begin{rem*}
We could replace property \ref{genfiltrn-prop-1} by
\begin{enumerate}
\renewcommand{\labelenumi}{\arabic{enumi}$'$.}
\item $1 \in F_1$, \label{genfiltrn-prop-1'}
\end{enumerate}
since (\oldref{genfiltrn-prop-1}$'$) and \ref{genfiltrn-prop-4} imply \ref{genfiltrn-prop-1}.
\end{rem*}

Filtration $\scrF$ is called {\em non-negative} if $F_i = 0$ for all $i < 0$. Associated to each filtration $\scrF$ there are two graded rings:
\begin{align*}
\af &:= \bigoplus_{i \in \zz} F_i\quad \text{and} \\
\gring &:= \bigoplus_{i \in \zz} (F_i/F_{i-1}).
\end{align*}
We denote a copy of $f \in F_d$ in the $d$-th graded component of $\af$ by $(f)_d$. $\af$ is given the structure of a graded $\kk$-algebra with multiplication defined by: $$(\sum_d(f_d)_d)(\sum_e(g_e)_e) := \sum_k\sum_{d+e = k}(f_dg_e)_k.$$
$\scrF$ is called a {\em finitely generated} \genfiltrn if $\profing$ is a finitely generated $\kk$-algebra.\\

Let $t$ be an indeterminate over $A$. Then there is an isomorphism
\begin{align}\label{t-isomorphism}
\af \cong \sum_{i \in \zz} F_it^i \subseteq A[t,t^{-1}]
\end{align}
which maps $(1)_1 \to t$. The following property of \profing\ is a straightforward corollary of this isomorphism.

\begin{lemma} \label{profing-domain-lemma}
\profing\ is an integral domain if and only if $A$ is an integral domain. \qed
\end{lemma}

\begin{fact*}
Let $S = \dsum_{i\geq 0} S_i$ be a graded ring such that $S_0 = \kk$ and also $S$ be an integral domain which is finitely generated as an algebra over $\kk$. Then $\proj S$ is a projective algebraic variety over $\kk$. \qed
\end{fact*}

It follows via the preceding fact that $\proj \af$ is a projective, and hence complete, variety if $\scrF$ is non-negative, finitely generated and $F_0 = \kk$. In view of this fact we make the following

\begin{defn*}
A filtration $\scrF = \{F_d\}_{d \in \zz}$ on $A$ is called {\em \completee} if it is \nonnegativee, finitely generated and $F_0 = \kk$.
\end{defn*}

The following proposition connects filtrations on $A$ with projective completions of $\spec A$.

\begin{prop}\label{prop-spec-completion}
If $A$ is a $\kk$-algebra and $\scrF$ is a \nonnegative filtration on $A$, then there is an open immersion $\psi_\scrF$ of $\spec A$ onto an open dense subscheme of $\proj \profing$. The complement of $\spec A$ in $\proj \profing$ is a hypersurface. Conversely, given any graded ring $S = \dsum_{i \geq 0} S_i$ and an open immersion $\phi: \spec A \into \proj S$ such that $\spec A$ is dense in $\proj S$ and $\proj S \setminus \spec A$ is a hypersurface, there is a \nonnegative filtration $\scrF$ on $A$ and a commutative diagram as follows:
\begin{center}\
$\begin{array}{c}
\xymatrix{
\spec A \ar[d]|{\psi_\scrF} \ar[r]^\cong 	& \spec A \ar[d]|\phi\\
\proj \profing \ar[r]			& \proj S}
\end{array}$
\end{center}
where the top horizontal map is the identity and the bottom horizontal map is a closed immersion induced by an isomorphism of $\profing$ with $S/I$ for some ideal $I$ of $S$ contained in the nilradical of $S$.
\end{prop}
\begin{rem*}
\mbox{}
\begin{itemize}
\item By a {\em hypersurface} in $\proj S$ for an arbitrary graded ring $S$ we just mean a closed subscheme given by $V(f)$ for some $f \in S$.
\item The first part of the theorem is well known. 
\end{itemize}
\end{rem*}

\begin{proof} Let us write $X$ for $\spec A$ and $Y$ for $\proj \profing$. Recall that basic open sets in $Y$ are given by $Y_G = \{Q \in Y | G \not \in Q\} = \spec \af_{(G)}$, where $G$ ranges over the homogenous elements in \profing, and $\profing_{(G)}$ is the subring of the local ring $\profing_{G}$ consisting of degree zero homogenous elements. Identify $\af$ with $\sum_{i \in \zz} F_it^i$ via \ref{t-isomorphism}. It follows that $\af_{(t)} = \{gt^d/t^d: g \in F_d \subseteq A,\ d \geq 0\} \cong \{g: g \in F_d \subseteq A\} = A$ as a subring of $A[t,t^{-1}]$; hence $X \cong Y_{t}$ and $Y\setminus X = V(t)$.\\

Now we prove that $Y_{t}$ is dense in $Y$. Let $G=(g)_d \in \profing$ be a homogenous element such that $Y_{t}\cap Y_G = \emptyset$, or equivalently, $V(t) \cup V(G) = \proj \af$. Then for all homogenous prime ideals $Q$ of \profing, either $t \in Q$, or $G \in Q$, so that
$$Gt \in \bigcap_{Q \in \proj \profing} Q \quad= \eta(\profing),$$
where $\eta(\profing)$ is the nilradical of $\profing$. (The last equality holds because the nilradical of a graded ring is a homogenous ideal and hence lies in the ideal $Q^h$ generated by homogenous elements of every prime ideal $Q$. But it is straightforward to see that if $Q$ is prime, then $Q^h$ is also prime.) Pick $k$ such that $(Gt)^k = (g^k)_{dk+k} = 0$. Then $g^k = 0 \in A$, and hence $G^k = (g^k)_{dk} = 0 \in \profing$. Therefore $G$ is nilpotent in $\profing$ and $Y_G = \emptyset$. Thus $Y_{t}$ is dense in $Y$, and the first assertion is proved.\\

We come to the last assertion of the theorem now. Let $S = \dsum_{i \geq 0}S_i$ be a graded ring, let $Z = \proj S$. Let $\phi: \spec A \to Z$ be an open immersion of $\spec A$ onto a dense open subscheme of $Z$ of the form $Z_f$ for some homogenous $f \in S$. Let the degree of $f$ be $d$. If we consider the $d$-th truncated ring $S^{[d]} \subseteq S$ defined by
$$S^{[d]} := \dsum_{d|k} S_k = \dsum_{k \geq 0} S_{kd},$$
then $\proj S \cong \proj S^{[d]}$, hence we can safely replace $S$ with $S^{[d]}$, and assume that degree of $f$ is $1$. Now, since $Z_f \cong \spec S_{(f)},$ and since morphisms of affine schemes are uniquely determined by the corresponding homomorphisms of rings (\cite{mumred}, theorem II.2.1), the above isomorphism of schemes induces an isomorphism of rings $\phi^*:S_{(f)} \cong A$. Now, for each $g \in A$, $(\phi^*)^{-1}(g) = a/f^k$ for some $k \geq 0$ and $a \in S_{k}$. Define
\begin{align*}
F_k := \phi^*(S_k/f^k) = \{g \in R|\ (\phi^*)^{-1}(g) \in S_k/f^k\}.
\end{align*}
Then it is easy to see that $\scrF = \{F_i\}_{i \geq 0}$ is a filtration on $A$. By means of this filtration we construct as usual the ring $\profing := \dsum_{d \geq 0}F_d$. Now $\phi^*: S_{(f)} \to A$ induces a mapping $\Phi^*:S \to \profing$. This map is defined componentwise, and it takes a $g \in S_d$ to $(\phi^*(g/f^d))_d$, i.e. to the copy of $\phi^*(g/f^d)$ in $\profing_d$. It is easily seen that $\Phi^*$ is a surjective homomomorphism of graded rings, so that the induced map $\Phi: \proj \profing \to \proj S$ is a closed immersion.

\begin{claim*} $\ker\Phi^* \subseteq \eta(S)$, where $\eta(S)$ is the nilradical of $S$. \end{claim*}

\begin{proof}
Pick $g \in S_d$ such that $\Phi^*(g) = (\phi^*(g/f^d))_d = 0 \in \profing$. Then $\phi^*(g/f^d) = 0 \in A$. Since $\phi^*: S_{(f)} \to A$ is an isomorphism, it follows that $g/f^d = 0 \in S_{(f)}$ so that there is some $k \geq 0$ such that $f^kg = 0 \in S$. But then every prime ideal in $S$ has to contain $f$ or $g$, ie $Z_f \cap Z_g = \emptyset$. Since $Z_f\ (\cong \spec A)$ is dense in $Z$, it follows that $Z_g = \emptyset$, so that $g \in \eta(S)$.
\end{proof}

By the above claim, $\bar S:= S/\ker\Phi^* \cong \profing$, and hence $\proj \profing \cong \proj \bar S$. To see that this isomorphism is the identity map when restricted to $\spec A$, note that the maps $\phi:\spec A \to \proj S$ and $\psi_\scrF:\spec A \to \proj \profing$ are completely determined by the corresponding maps $\phi^*: S_{(f)} \to A$ and $\psi_\scrF^*:\profing_{(t)} \to A$, where $t = (1)_1 \in \profing$ as before. But the latter two maps of the rings give rise to the following commutative diagram:
\begin{center}
$\begin{array}{cc}
\xymatrix{
A \ar[r]^{1_A} 	& A\\
S_{(f)} \ar[u]|{\phi^*} \ar[r]^{\Phi^*} & \profing_{(t)} \ar[u]|{\psi_\scrF^*}} &
\xymatrix{
\phi^*(g/f^d) \ar@{|->}[r]^{1_A} 	&  \phi^*(g/f^d)\\
g/f^d \ar@{|->}[u]|{\phi^*} \ar@{|->}[r]^{\Phi^*} & \frac{(\phi^*(g/f^d))_d}{t^d} \ar@{|->}[u]|{\psi_\scrF^*}}
\end{array}$
\end{center}

Since the the top horizontal map and both of the vertical maps are isomorphisms, it follows that the bottom horizontal map must be an isomorphism too, and thus we get the desired commutative diagram.
\end{proof}

Before we proceed further, let us examine the structure of the complement of $\spec A$ in $\proj \profing$ more closely. From the proof of proposition \ref{prop-spec-completion}, we see that $\proj \profing \setminus \spec A$ is the closed subspace $V(t) := \{P \in \proj \profing |\ t \in P\}$. Now consider the natural projection $\pi: \profing \to \gring$. Homomorphism $\pi$ of graded rings is surjective and therefore it defines a map $\pi^*:\proj(\gring)\to \proj \profing$, given by $p \to \pi^{-1}(p)$, which is a closed immersion \cite[Exercise II.3.12(a)]{hart}.

\begin{claim*}$\ker\pi = \langle t \rangle$. \end{claim*}

\begin{proof}
Since  $t := (1)_1$ and  $1 \in F_0$ , it follows that $t \in \ker\pi$, so that $\ker\pi \supseteq \langle t \rangle.$ Also, $\pi$ a homomorphism of graded ring. Hence $\ker\pi$ is a homogenous ideal, i.e. $\ker\pi$ is generated by homogenous elements from $\profing$. Pick a homogenous $G$ in the $d$-th graded component of $\ker\pi$. Then $G = (g)_d$ for some $g \in F_{d-1}$. But then
$$G = (g)_d = (g)_{d-1}(1)_1 = (g)_{d-1}t,$$
so that $G \in \langle t \rangle$. Therefore, $\ker\pi \subseteq \langle t \rangle$.
\end{proof}

It follows that $\pi^*(\proj(\gring)) = V(t)$ \cite[Exercise II.3.12(b)]{hart}. Thus we have proven the following
\begin{prop} \label{gradedprop}
 $\proj \profing \setminus \spec A = \proj(\gring)$. \qed
\end{prop}

Combining the results of propositions \ref{profing-domain-lemma}, \ref{prop-spec-completion} and \ref{gradedprop}, we deduce

\begin{thm}\label{completethm}
Let $A$ be a $\kk$-algebra which is an integral domain and $\scrF$ be a filtration of $A$. If $\scrF$ is \completee, then $\proj \profing$ is a projective variety over $\kk$ and the map
$$\psi_\scrF:p \to \dsum_{d\geq 0}p \cap F_d$$
gives an open immmersion of $\spec A$ onto a dense open subscheme of $\proj \profing$. The complement of $\spec A$ in $\proj \profing$ is the hypersurface $V((1)_1)$, which is also the image of the closed immersion $\pi^*:\proj(\gring) \into \proj(\profing)$ induced by the natural projection $\pi: \profing \to \gring$.  Conversely, given any graded ring $S$ and an open immersion $\phi: \spec A \into \proj S$ such that $\spec A$ is dense in $\proj S$ and $\proj S \setminus \spec A$ is a hypersurface, there is a filtration $\scrF$ on $A$ and a commutative diagram as follows:
\begin{center}\
$\begin{array}{c}
\xymatrix{
\spec A \ar[d]|{\psi_\scrF} \ar[r]^\cong 	& \spec A \ar[d]|\phi\\
\proj \profing \ar[r]			& \proj S}
\end{array}$
\end{center}
where the top horizontal map is the identity and the bottom horizontal map is a closed immersion which induces an isomorphism $\proj \profing \cong \proj (S/\eta(S))$ where $\eta(S)$ is the nilradical of $S$. \qed
\end{thm}

Recall that the {\em homogenous coordinate ring} of a closed subvariety of a $n$-dimensional weighted projective space is the factor of the polynomial ring of $n+1$ variables by the ideal of all weighted homogeneous polynomials vanishing on this variety. Choosing a set of generators of $\af$ gives an embedding of $\proj\af$ into a weighted projective space such that the homogeneous coordinate of the image is $\af$. The following corollary is a rewording of theorem \ref{completethm} in terms of these coordinates.

\sloppy

\begin{cor}\label{completecor}
Let $X$ be an affine variety over $\kk$ and $A$ be its coordinate ring. Let $\scrF$ be a \complete filtration on $A$ such that $\profing$ is generated over $\kk$ by $(f_1)_{d_1}, \ldots,  (f_N)_{d_N}$. Then the map $\phi:X \to \aNk$ given by $\phi(x) := (f_1(x), \ldots, f_N(x))$ is a closed immersion, and there exists a commutative diagram as follows:
\begin{center}
$\begin{array}{c}
\xymatrix{
X \ar[d]_{\psi_\scrF} \ar[r]^{\phi} 	& \aNk \ar[d]^i\\
\proj \profing \ar[r]^-\Phi 		& \pp^N(\kk;1,d_1,\ldots,d_N)}
\end{array}$
\end{center}
where $i$ is the natural injection of $\aNk$ into the weighted projective space $\pp^N(\kk;1,d_1,\ldots,d_N)$, and $\Phi$ is a closed immersion that maps $\proj \profing$ isomorphically onto $\bar X$, the closure of $X$ in $\pp^N(\kk;1,d_1,\ldots,d_N)$. Moreover, $\Phi$ induces an isomorphism between $\profing$ and the homogenous coordinate ring of $\bar X$. Conversely, given any closed immersion $\phi:X \into \aNk$, and any $N$ positive integers $d_1, \ldots, d_N$, there exists a \complete filtration $\scrF$ on $A$ and a closed immersion $\Phi: \proj \profing \into \pp^N(\kk;1,d_1,\ldots,d_N)$ such that the above diagram is commutative. Also, $\Phi$ induces a graded ring isomorphism between the graded ring of $\bar X \subseteq \pp^N(\kk;1,d_1,\ldots,d_N)$ and $\profing$, and hence an isomorphism of $\proj \profing$ and $\bar X$. \qed
\end{cor}

\fussy
\subsection{Examples} \label{subsec-filtrintro-examples}
Now we look at some examples of \complete filtrations on the coordinate rings of affine varieties, and find the corresponding completions. Throughout this section $X$ will denote an affine variety, $A$ will denote its coordinate ring, $\scrF = \{F_d: d \geq 0\}$ will be a filtration on $A$ with the corresponding completion $\xf$ of $X$, and $\af$ will denote the homogenous coordinate ring of $\xf$. For the first four examples we set $X = \ank$ and $A = \kk[x_1,\ldots,x_n]$.

\begin{example} Define $\scrF$ on $A$ by letting $F_d$ be the set of polynomials of degree less than or equal to $d$. Identification of $F_d$ with the set of {\em homogenous} polynomials in $x_0, \ldots, x_n$ of degree $d$ gives an isomorphism $\af \cong \kk[x_0,\ldots,x_n]$. Hence $\xf$ is the usual projective space $\pp^n(\kk)$.
\end{example}

\begin{example} Let $d_1, \ldots, d_n$ be any $n$ positive integers. Let $F_d$ be the $\kk$-linear span of all the monomials $x_1^{\alpha_1}x_2^{\alpha_2}\cdots x_n^{\alpha_n}$ such that $\sum \alpha_id_i \leq d$. Set $d_0 := 1$. Then $F_d$ can be identified with the set of {\em weighted homogenous} polynomials in $x_0, \ldots, x_n$ of of weighted degree $d$ and with weights of $x_i$ being $d_i$. Thus $\af$ is again isomorphic to $\kk[x_0,\ldots,x_n]$, but the grading in this case is induced by the weighted degree $(d_0,d_1,\ldots, d_n)$, and $\xf$ is the weighted projective space $\pp^n(\kk;d_0,d_1,\ldots,d_n)$.
\end{example}

\begin{example} Let $F_i$ be the set of polynomials of degree less than or equal to $di$, where $d$ is a fixed positive integer. Then $\xf$ is the $d$-uple embedding of $\pp^n(\kk)$ in $\pp^m(\kk)$, where $m = \binom{n+d-1}{n-1}$. In particular, for $n=1$, $\xf$ is the rational canonical curve of degree $d$ in $\pp^d(\kk)$.
\end{example}

\begin{example}
$X$ is again $\ank$ as above. Let $F_1$ be the $\kk$-linear span of all monomials of degree less than or equal to two except for $x_n^2$. Let $F_d = (F_1)^d$ for $d \geq 1$. Then $\xf$ is isomorphic to the variety resulting from a blow up of $\pp^n(\kk)$ at the point $O := [0:\cdots:0:1]$. To see this, let $Y$ be the blow up of $\pp^n(\kk)$ at $O$. Then $Y$ is the closure in $\pp^n(\kk) \times \pp^{n-1}(\kk)$ of the image of $\ank$ under the map $(x_1,\ldots,x_n) \mapsto ([1:x_1:\cdots:x_n],[1:x_1:\cdots:x_{n-1}])$. Embedding $\pp^n(\kk) \times \pp^{n-1}(\kk)$ into $\pp^{n(n+1)-1}(\kk)$ via the Segre map, we see that $Y$ is isomorphic to the closure in $\pp^{n(n+1)-1}(\kk)$ of the image of $X$ under the map $(x_1,\ldots,x_n) \mapsto [1:x_1: \cdots:x_{n-1}:x_1: x_1^2: \cdots : x_1x_{n-1}: \cdots : x_n: x_nx_1: \cdots: x_nx_{n-1}]$. Composing with an automorphism of $\pp^{n(n+1)-1}(\kk)$, we may assume the map is $(x_1,\ldots,x_n) \mapsto [1:x_1: \cdots:x_{n-1}:x_1^2: x_1x_2:\cdots : x_1x_{n-1}: (x_2)^2:x_2x_3:\cdots:x_2x_{n-1}:\cdots : x_{n-1}^2: x_{n-1}x_n:0:\cdots:0]$. Projecting onto the first $N := n + (n-1) + \cdots + 2$ coordinates, we see that $Y$ is isomorphic to the closure in $\pp^{N-1}(k)$ of the map of $\cc^n$ given by all monomials of degree less than or equal to two except for $x_n^2$. But this is precisely $\xf\quad $!       
\end{example}

\begin{example}
Let $X$ be a normal affine variety with trivial divisor class group $\cl X$ (e.g. any $X$ whose coordinate ring is a unique factorization domain). Let $\bar{X} \subseteq \pp^L(\kk)$ be any normal projective completion of $X$. If $X_\infty := \bar{X}\setminus X$ is irreducible, then the embedding $X \into \bar{X}$ arises from a filtration. To see this, let $D$ be the Cartier divisor on $\bar{X}$ such that $\sheaf_{\bar{X}}(D)$ is the line bundle $\sheaf_{\bar{X}}(1)$. Since $\cl X = 0$, there is $f \in \kk(X)$ such that $(f) = D|_{X}$. Let $D' = D - (f)$, so that $\supp(D') \cap X = \emptyset$. It follows that $D' = kX_\infty$ for some $k \in \zz$. But $k = \deg D' = \deg D = \deg \bar{X} > 0$ \cite[Excercise II.6.2]{hart}, and hence $1 \in \Gamma(\sheaf_{\bar X}(D'),\bar{X})$. Without loss of generality we may assume that $\bar{X}$ does not lie in any proper hyperplane in $\pp^L(\kk)$. Then $\dim_\kk \Gamma(\sheaf_{\bar X}(D'),\bar{X}) = L+1$. Let $g_1, \ldots, g_L \in \kk(X)$ such that $1, g_1, \ldots, g_L$ is a vector space basis of $\Gamma(\sheaf_{\bar X}(D'),\bar{X})$. Since $(g_i) + D' \geq 0$ for all $i = 1, \ldots, L$, it follows that none of the $g_i$'s has any pole on $X$, and hence each $g_i \in A$. Let $\phi:X \to \affine[L]{\kk}$ be the map which takes $x \in X$ to $(g_1(x), \ldots, g_L(x))$. Then $\bar{X}$ is precisely the closure of $\phi(X)$ in $\pp^L(\kk)$. By theorem \ref{completethm} it follows that the embedding $X \into \bar{X} \subseteq \pp^L(\kk)$ arises from a filtration on $A$.   
\end{example} 

\begin{example}
Not all projective completions of affine varieties are determined by a filtration. Our example below is a variation of an example by Mike Roth and Ravi Vakil considered in \cite{roth-vakil-affine}. Let $X'$ be a nonsingular cubic curve in $\pp^2$. Let $O$ be one of its $9$ inflection points. Consider the group structure on $X'$ with $O$ as the origin. Pick any point $P$ of $X'$ which is not a torsion point in this group. Then $X := X' \setminus \{P\}$ is an affine variety (\cite[Proposition 5]{goodman-affine}). We claim that there is no homogenous polynomial $f$ in $\cc[x_0, x_1, x_2]$ such that $V(f) \cap X' = \{P\}$. Indeed, assume to the contrary that there is a homogenous polynomial $f$ of degree $d > 0$ such that $V(f) \cap X' = \{P\}$. Then $(f) = 3dP$, where $(f)$ is the divisor on $X'$ associated to $f$. Let $l$ be the equation of the tangent line of $X'$ at $O$. Since $O$ is an inflection point, it follows that $V(l) \cap X' = \{O\}$ and thus the divisor of $l$ on $X'$ is $(l) = 3O$. Then $f/l^d$ is a rational function on $X'$, and its corresponding divisor on $X'$ is $(f/l^d) = 3dP - 3dO$. Since $(f/l^d)$ is a principal divisor, it follows that $(f/l^d) \equiv 0$ in the group structure of $X'$. But then $3dP \equiv 0$ and hence $P$ is a torsion point, which is a contradiction. Thus the claim is true and there is no homogenous polynomial $f$ such that $V(f) \cap X' = \{P\}$. But due to proposition \ref{prop-spec-completion}, there is no integer $d>0$ such that the embedding of $X$ into the image of the $d$-uple embedding of $X'$ comes from a filtration!
\end{example}

\subsection{Two Existence Theorems} \label{subsec-filtrintro-existence}
Recall that for finitely many closed subvarieties $V_1, \ldots, V_m$ of $X$, a completion $\psi: X \into Z$ is said to {\em preserve the intersection of $V_1, \ldots, V_m$ at $\infty$} if $\bar{V}_1 \cap \cdots \cap \bar{V}_m \cap X_\infty =  \emptyset$, where $X_\infty := Z\setminus X$ is the set of `points at infinity' and $\bar{V}_j$ is the closure of $V_j$ in $Z$ for every $j$.

\begin{thm} \label{filtrexistence-thm1}
Let $V_1, \ldots, V_m$ be Zariski closed subsets in an affine variety $X$ such that $\cap_{i=1}^m V_i$ is a finite set. Then there is a complete filtration $\scrF$ on $A$ such that $\psif$ preserves the intersection of the $V_i$'s at $\infty$.
\end{thm}

\begin{proof}
Let $A = \kk[x_1,\ldots,x_n]/p$, where $p$ is the prime ideal defining $X \subseteq \ank$. Let $V_i = \spec(\kk[x_1,\ldots,x_n]/q_i)$ with $q_i \supseteq p$ for each $i$.

\begin{claim*}
For each $i = 1, \ldots, n$, there is an integer $d_i \geq 1$ such that
\begin{align}
x_i^{d_i} &= f_{i,1} + \ldots + f_{i,m} + g_i \label{filtrexistence1-eqn1}
\end{align}
for some $f_{i,j} \in q_j$ and a polynomial $g_i \in \kk[x_i]$ of degree less that $d_i$.
\end{claim*}

\begin{proof}
If $V_1 \cap \ldots \cap V_m = \emptyset$, then by Nullstellensatz $\langle q_1, \ldots, q_m \rangle$ is the unit ideal in $\kk[x_1, \ldots, x_n]$, and the claim is obvious. So assume $$V_1 \cap \ldots \cap V_m = \{P_1, \ldots, P_k\} \subseteq \ank,$$
for some $k \geq 1$. Let $P_i = (a_{i,1}, \ldots, a_{i,n}) \in \ank$. For each $i = 1, \ldots, n$, let
$$h_i := (x_i - a_{1,i})(x_i - a_{2,i}) \cdots (x_i - a_{k,i}).$$
By Nullstellensatz, for some $d'_i \geq 1$, $h_i^{d'_i} \in \langle q_1, \ldots, q_m \rangle$, ie $h_i^{d'_i} = f_{i,1} + \ldots + f_{i,m}$ for some $f_{i,j} \in q_j$. Substituting $h_i = \prod_j(x_i - a_{j,i})$ in the preceding equation we see that the claim holds with $d_i := kd'_i$.
\end{proof}

Below  for  $S \subseteq \kk[X]$ we denote by $\kk\langle S \rangle$ the $\kk$-linear span of $S$.\\

Fix a set of $f_{i,j}$'s satisfying the conclusion of the previous claim. Then define a filtration $\scrF$ on $A$ as follows: let
\begin{align*}
	F_0 &:= \kk, \\
	F_1 &:= \kk\langle 1, \bar x_1, \ldots, \bar x_n, \bar f_{1,1}, \ldots, \bar f_{n,m} \rangle, \\
	F_k &:= F_1^k\ \textrm{for}\ k>1,\\
	\scrF &:= \{F_i|\ i\geq 0\}.
\end{align*}

Clearly $\scrF$ is a complete filtration. We now show that this $\scrF$ satisfies the conclusion of the theorem. Let $\psi_\scrF: \spec A \into \proj \profing$ be the embedding. Recall that for $\bar q \in \spec A$,
$$\psi_\scrF(\bar q) = \dsum_{i \geq 0}\bar q \cap F_i.$$
Since $V_i = V(\bar q_i) \subseteq \spec A$, it follows that the closure of $V_i$ in $\proj \profing$ is $\bar V_i = V(\psi_\scrF(\bar q_i)) \subseteq \proj \profing$. Since $\proj \profing \setminus \spec A = V((1)_1)$, we have to show that
$$V(\psi_\scrF(\bar q_1), \ldots, \psi_\scrF(\bar q_m), (1)_1) = \emptyset \in \proj \profing.$$
But this is equivalent to showing that $\sqrt I = \profing_+$, where $I$ is the ideal generated by $\psi_\scrF(\bar q_1), \ldots, \psi_\scrF(\bar q_m)$ and $(1)_1$ in $\profing$, and where $\profing_+ := \dsum_{i>0}F_i$ is the so called {\em irrelevant} ideal of $\profing$.\\

\sloppy

From the construction of $\scrF$ it follows that $\profing_+$ is generated by the elements $(1)_1, (\bar x_1)_1, \ldots, (\bar x_n)_1, (\bar f_{1,1})_1, \ldots, (\bar f_{n,m})_1$. Note that $\bar f_{i,j} \in \bar q_j$ for each $i,j$, so that $(\bar f_{i,j})_1 \in \psi_\scrF(\bar q_j) \subseteq I$. Moreover, $(1)_1 \in I$. So, all we really need to show is that $(\bar x_i)_1 \in \sqrt I$ for all $i = 1, \ldots, n$.\\

\fussy

Consider equation \ref{filtrexistence1-eqn1} and its reduction mod $p$, i.e.
$$(\bar x_i)^{d_i} = \bar f_{i,1} + \ldots + \bar f_{i,m} + \bar g_i \in A$$
for all $i = 1, \ldots, n$. Let $g_i = \sum_{j=0}^{d_i-1}a_{i,j}(x_i)^j$. Then in $\profing$,
$$((\bar x_i)_1)^{d_i} = ((1)_1)^{d_i-1}((\bar f_{i,1})_1 + \ldots + (\bar f_{i,m})_1) + \sum_{j=0}^{d_i-1}a_{i,j}((\bar x_i)_1)^j((1)_1)^{d_i-j}.$$
All of the summands in the right hand side lie inside $I$, hence $((\bar x_i)_1)^{d_i} \in I$ for all $i = 1, \ldots, n$, as required.
\end{proof}
Recall that given a polynomial map $f=(f_1, \ldots, f_q):X \to \cc^q$, $a = (a_1, \ldots, a_q) \in f(X)$ and a completion $\psi$ of $X$, $\psi$ {\em preserves a fiber $\finv(a)$ at $\infty$} if $\psi$ preserves the intersection of the hypersurfaces $H_i(a) := \{x \in X: f_i(x) = a_i\}$, $i = 1, \ldots, q$.

\begin{example}\label{filtrexistence-example1}
Consider $f:\affine[2]{\kk} \to \affine[2]{\kk}$ given by $f(x,y) = (x, y + x^3)$. For $a = (a_1, a_2) \in \kk^2$, 
\begin{align*}
H_1(a) &= \{(a_1, y): y \in \kk\},\\
H_2(a) &= \{(x,a_2-x^3):x \in \kk\}.
\end{align*}
In the usual completion $\pp^{2}(\kk)$ of $\kk^2$, the closures of $H_1(a)$ and $H_2(a)$ intersect at the point $(0:0:1)$ at infinity for each $a \in \kk^2$, and hence $\pp^{2}(\kk)$ does not preserve any fiber of $f$ at $\infty$.\\

We now produce in the same way as suggested by theorem \ref{filtrexistence-thm1} a completion of $\affine[2]{\kk}$ which preserves $\finv(0)$ at infinity. In the notation of theorem \ref{filtrexistence-thm1}, $q_1	= \langle x \rangle$ and $q_2	= \langle y+x^3 \rangle$. Observe that $x \in q_1$, and $y$ satisfies
\begin{align*}
y =  - x^3 + (y+x^3),
\end{align*}
with $x^3 \in q_1$ and $y+x^3 \in q_2$. Then as in theorem \ref{filtrexistence-thm1}, we define the filtration $\scrF:= \{F_i|\ i\geq 0\}$ on $\kk[x,y]$ by: $F_0 := \kk$, $F_1 := \kk\langle 1, x, y, x^3 \rangle$, and $F_k := (F_1)^k$ for $k>1$. By theorem \ref{filtrexistence-thm1}, it follows that $\psi_\scrF$ preserves $\finv(0)$ at $\infty$. In fact, the corresponding completion $\xf = V(w^2z - x^3) \subseteq \pp^3(\kk)$, where the coordinates of $\pp^3(\kk)$ are $[w:x:y:z]$, and it is easy to show that $\xf$ preserves {\em every} fiber of $f$ at $\infty$.
\end{example}

\begin{example}\label{filtrexistence-example2}
Let $f(x,y) = (x,y)$, i.e. $f$ is the identity mapping of $\kk^2$. Then for $a = (a_1, a_2) \in \kk^2$,
\begin{align*}
H_1(a) &= \{(a_1,y)|\ y \in \kk\}, \\
H_2(a) &= \{(x,a_2)|\ x \in \kk\}.
\end{align*}
Consider the filtration $\scrF$ on $\kk[x,y]$ defined by: $F_0 := \kk$, $F_1 := \kk\langle 1, x, y, xy, x^2y^2 \rangle$, and $F_k := (F_1)^k$ for $k \geq 2$. By corollary \ref{completecor}, $\xf$ is the closure of the image of $\affine[2]{\kk}$ under the map $\psi:\affine[2]{\kk} \into \pp^4(\kk)$ defined by:
$$\psi(x,y) = (1:x:y:xy:x^2y^2).$$
But then $\psi(H_1(a)) = \{(1:a_1:y:a_1 y:a_1^2y^2)|\ y \in \kk)\}$. If $a_1=0$, then $\psi(H_1(a)) = \{(1:0:y:0:0)|\ y \in \kk)\}$, and hence the only point at infinity in $\overline{\psi(H_1(a))}$ is $(0:0:1:0:0)$. But if $a_1 \neq 0$, the point at infinity in $\overline{\psi(H_1(a))}$ is $(0:0:0:0:1)$. Similarly, $\psi(H_2(a)) = \{(1:x:a_2:y:a_2 x:a_2^2x^2)|\ x \in \kk)\}$ and the only point at infinity in $\overline{\psi(H_2(a))}$ is $(0:1:0:0:0)$ if $a_2=0$, and $(0:0:0:0:1)$ if $a_2 \neq 0$. Therefore $\psi$ preserves $\finv(a)$ at $\infty$ iff $a$ belongs to the union of the coordinate axes.
\end{example}

Let $f=(f_1,\ldots,f_r): X \to Y \subseteq \kk^r$ be a dominating map of affine varieties of dimension $n$. Given any $y \in Y$ such that $\finv(y)$ is finite, theorem \ref{filtrexistence-thm1} guarantees the existence of a projective completion of $X$ that preserves $f^{-1}(y)$ at $\infty$. But as the preceding example shows, it might be the case that the completion fails to preserve at $\infty$ `most of the' fibers of $f$. This suggests that we should look for a completion $\psi$ which preserves $\finv(y)$ at $\infty$ for {\em generic}  $y \in Y$, in which case, recall that we say $\psi$ {\em preserves $f$ at $\infty$}. The following theorem says that there exist filtrations on $\kk[X]$ which give rise to such completions.

\begin{thm}\label{filtrexistence-thm2}
Let $f: X \to Y$ be a dominating morphism of affine varieties of dimension $n$. Then there is a complete filtration $\scrF$ on the coordinate ring of $X$ such that $\psi_\scrF$ preserves $f$ at $\infty$.
\end{thm}

\begin{proof}
By Noether normalization, there is a finite map $f':Y \to \ank$. Composing $f$ with $f'$ if necessary, we can assume $Y = \ank$. Let $\kk[X] = \kk[x_1,\ldots,x_m]/p$, where $p$ is the prime ideal defining $X \subseteq \affine[m]{\kk}$. Let the coordinates of $\ank$ be $y_1,\ldots,y_n$. Since $f$ is dominating, it follows that $f^*:\kk[y_1,\ldots,y_n] \to \kk[X]$ is an injection. Since dimension of $X$ is $n$, it then implies that $\kk[X]$ is algebraic over $\kk[y_1,\ldots,y_n]$.\\ 

In order to simplify notations in the proof below we identify $x_i$'s with their classes $\mod p$ (and do not put the `bar' over the $x_i$'s). We also use $y_j$'s and $f_j$'s interchangeably, since they define the same element in $\kk[X]$.\\

Since $\kk[X]$ is algebraic over $\kk[y_1,\ldots,y_n]$, it follows that for all $i = 1, \ldots, m$, 
\begin{align}
\sum_{j=0}^{k_i}g_{i,j}(y)(x_i)^{k_i-j} &= 0 \label{app1-eqn2}
\end{align}
for some $k_i \geq 1$ and $g_{i,j}(y) \in \kk[y_1,\ldots,y_n]$ such that $g_{i,0} \neq 0$. Let $g_{i,j}(y) = \sum_{\alpha}c_{i,j,\alpha}y^{\alpha}$. Consider the partial ordering $\prec$ on $\nn^n$ given by $\alpha \prec \beta$ iff $\beta - \alpha \in \nn^n$. Define a filtration $\scrF := \{F_i|\ i\geq 0\}$ on $\kk[X]$ as follows: let
\begin{eqnarray*}
	F_0 &:=& \kk, \\
	F_1 &:=& \kk\langle 1, x_1, \ldots, x_m, y_1, \ldots, y_n \rangle \\
			&	 & +\ \kk\langle y^{\beta}|\ \beta \prec \alpha\ \textrm{for some $\alpha$ such that}\ c_{i,k_i,\alpha} \neq 0 \rangle \\
	    &	 & +\ \kk\langle x_iy^{\beta}|\ \beta \prec \alpha\ \textrm{for some $\alpha$ such that}\ c_{i,k_i-1,\alpha} \neq 0 \rangle, \\
	F_k &:=& \left\{
		\begin{array}{l}
		\sum_{j=1}^{k-1} F_jF_{k-j} + \kk\langle (x_i)^ky^{\beta}|\ \beta \prec \alpha\ \textrm{for some $\alpha$ such that}\ c_{i,k_i-k,\alpha} \neq 0 \rangle \\
	    	\textrm{for}\ 1<k \leq \max\{k_1,\ldots,k_m\},\ \textrm{and} \\
		\sum_{j=1}^{k-1} F_jF_{k-j} \quad \textrm{if}\ k>k_i\ \forall i.
		\end{array}
		\right.
\end{eqnarray*}
\renewcommand{\filtrationring}{\kk[X]}

\begin{claim*}
$\tilde X := \proj \profing$ satisfies the required property.
\end{claim*}

\begin{proof}
Writing $y=(y-b)+b$ in equation \ref{app1-eqn2}, expanding the polynomials, and then separating the terms that do not depend on any $y_j - b_j$, we derive
\begin{align}
\sum_{j=0}^{k_i}g_{i,j}(b)(x_i)^{k_i-j} + \sum_{j=0}^{k_i}\tilde h_{i,j}(b)\tilde g_{i,j}(y-b)(x_i)^{k_i-j} &= 0, \label{app1-eqn3}
\end{align}
where $g_{i,0} \neq 0$ and each monomial in each $\tilde g_{i,j}$ has degree greater than or equal to one. Let 
$$g := \prod_{i=1}^m g_{i,0}(y_1, \ldots, y_n) \in A.$$  
Let $V :=  \{a \in Y: g(a) \neq 0\} \cap f(X)$. Since $f$ is dominating, $V$ is a non-empty Zariski open set of $f(X)$. Let $a:=(a_1, \ldots, a_n) \in V$. It suffices to prove that $\psi_\scrF$ preserves $\finv(a)$ at infinity. Let $H_j := \{x \in X: f_j(x) = a_j\}$ and let $q_j$ be the ideal of $H_j$, i.e. the ideal of $\kk[X]$ generated by $y_j - a_j$. Recall that the corresponding ideal $q_j^\scrF$ of $\profing$ is given by: $q_j^\scrF := \dsum_{i \geq 0} (q_j \cap F_i)$. Let $\tilde q$ be the ideal of $\profing$ generated by $q^\scrF_1, \ldots, q^\scrF_n$ and $(1)_1$. Then we have to show that
$$\sqrt{\tilde q} = \profing_+,$$ 
where $\profing_+ = \dsum_{i>0}F_i$ is the maximal homogenous ideal of $\profing$.\\

\sloppy

By construction of $\scrF$, it follows that $\profing_+$ is generated by elements $(1)_1, (x_1)_1, \ldots, (x_m)_1, (y_1)_1, \ldots, (y_n)_1$, the $(y^{\beta})_1$'s given in the definition of $F_1$, and all those $((x_i)^ky^\beta)_k$'s that we inserted in the definition of all $F_k$'s. We will show that some power of each of these generators lie in $\tilde q$.\\

\fussy

At first observe that for each $j$, $(y_j - a_j)_1 \in q_j^\scrF$. But 
$$(y_j)_1 = (y_j - a_j)_1 + a_j(1)_1,$$
so that each $(y_j)_1 \in \tilde q$. Similarly each $(y^\beta)_1$ is a $\kk$-linear combination of $((y - a)^\gamma)_1$'s for $\gamma \prec \beta$. Since each of the latter terms lie in $\tilde q$, each such $(y^\beta)_1$ is also in $\tilde q$. Finally, let $k \geq 1$, and consider $(x_i)^ky^\alpha \in F_k$. For {\em each} $\beta \in \nn^n$ such that $0 \neq \beta \prec \alpha$, $(x_i)^k(y-a)^\beta$ is also in $F_k$, since $(x_i)^k(y-a)^\beta$ is a $\kk$-linear combination of $(x_i)^ky^\gamma$ for $\gamma \prec \beta$, and all these $(x_i)^ky^\gamma$'s already lie in $F_k$. Pick any $j$ such that $\beta_j > 0$. Then $(x_i)^k(y-a)^\beta$ lies in $q_j$, and therefore $((x_i)^k(y-a)^\beta)_k \in q_j^\scrF \subseteq \tilde q$. But $(x_i)^ky^\alpha$ is a $\kk$-linear combination of such $(x_i)^k(y-a)^\beta$'s and $(x_i)^k$. It follows that $((x_i)^ky^\alpha)_k$ is a $\kk$-linear combination of such $((x_i)^k(y-a)^\beta)_k$'s and $((x_i)^k)_k = ((x_i)_1)^k$. Therefore, to show that $((x_i)^ky^\alpha)_k \in \sqrt{\tilde q}$, we only have to show that each $(x_i)_1 \in \sqrt{\tilde q}$.\\

Consider equation \ref{app1-eqn3}. At first note that by the same argument as in the preceding paragraph, for each $b \in \ank,\ \tilde g_{i,j}(y-b)(x_i)^{k_i-j} \in F_{k_i-j} \subseteq F_{k_i}$ for each $i,j$. Regarding each summand as an element in $F_{k_i}$, equation \ref{app1-eqn3} evaluated at $b = a$ gives, in $\profing$,
{\small
\begin{align*}
g_{i,0}(a)((x_i)_1)^{k_i} = - \sum_{j=1}^{k_i}g_{i,j}(a)((x_i)_1)^{k_i-j}((1)_1)^j \ - 
        	 \sum_{j=0}^{k_i}\tilde h_{i,j}(a)(\tilde g_{i,j}(y-a)(x_i)^{k_i-j})_{k_i}.
\end{align*}}

For each $j \geq 1$, $((x_i)_1)^{k_i-j}((1)_1)^j \in \langle (1)_1 \rangle \subseteq \tilde q$. Pick any nonzero term $a_\alpha(y-a)^\alpha$ of any $\tilde g_{i,j}(y-a)$ and pick any $j$ such that $\alpha_j > 0$. Then $a_\alpha(y-a)^\alpha(x_i)^{k_i-j} \in q_j$, and hence $(a_\alpha(y-a)^\alpha(x_i)^{k_i-j})_{k_i} \in q_j^\scrF \subseteq \tilde q$. Thus each term of the right hand side of the above equation lies in $\tilde q$. Since $a \in V$, $g_{i,0}(a) \neq 0$. It follows that $((x_i)_1)^{k_i} \in \tilde q$ for each $i$, as required. 
\end{proof}
\renewcommand{\qedsymbol}{}
\end{proof}

\begin{rem}
Let $X \subseteq \cc^m$. Trivial filtrations like the one induced by the projective closure of the graph in $\cc^{m+n}$ of map $f$ do not necessarily preserve $P$ at $\infty$, as the following example shows. Let $X = Y = \affine[2]{\cc}$ and $P:X \to Y$ be a quasifinite map defined by $f_1 := x_1^3 + x_1^2x_2 + x_1x_2^2 - x_2$, and $f_2 := x_1^3 + 2x_1^2x_2 + x_1x_2^2 - x_2$. Let the coordinates of $\pp^2$ be $[Z:X_1:X_2]$ where $x_i = X_i/Z$ for $i = 1,2$. Choose local coordinates $\xi_1 := X_1/X_2, \xi_2 := Z/X_2$ near $P := [0:0:1]$. Let $(a_1, a_2) \in \cc^2$. Equations of $f_{i,a} := f_i(x) - a_i$ in $(\xi_1,\xi_2)$ coordinates are:
\begin{align*}
f_{1,a} &= \xi_1^3 + \xi_1^2 + \xi_1 - \xi_2^2 - a_1\xi_2^3 \\
f_{2,a} &= \xi_1^3 + 2\xi_1^2 + \xi_1 - \xi_2^2 - a_1\xi_2^3
\end{align*}
It follows that for each $i$, $f_{i,a} = 0$ has a branch near $P$ with parametrization:
$$\gamma_{i,a}(t) := [t:t^2 +\ \text{h.o.t}:1],$$
where {\em h.o.t.} denotes higher order terms in $t$. In $(x_1, x_2)$ coordinates the parametrization of $\gamma_{i,a}$ becomes: $\gamma_{i,a}(t) := (t +\ \text{h.o.t}, 1/t)$.\\

Let $\Gamma$ be the graph of $P$ in $\cc^4$. Choose coordinates $(x_1, x_2, y_1, y_2)$ of $\cc^4$ such that map $\phi: X \to \Gamma$ is given by $\phi(x) = (x_1,x_2, f_1(x), f_2(x))$. Let $X'$ be the closure of $\Gamma$ in $\pp^4(\cc)$ and let the coordinates of $\pp^4(\cc)$ be $[Z:X_1:X_2:Y_1:Y_2]$ where $x_i = X_i/Z$ and $y_i = Y_i/Z$ for $i = 1,2$. Let $a := (a_1, a_2) \in \cc^2$. Then $f_2(\gamma_{1,a}(t)) = (t +\ \text{h.o.t})^3 + \frac{2}{t}(t +\ \text{h.o.t})^2 + \frac{1}{t^2}(t +\ \text{h.o.t}) - \frac{1}{t} = c_1 + \text{h.o.t}$, for some constant $c_1 \in \cc$. But then
\begin{align*}
\phi(\gamma_{1,a}(t)) &= [1:t +\ \text{h.o.t}\ : 1/t: a_1: c_1 + \text{h.o.t}] \\
											&= [t:t^2+\ \text{h.o.t}\ : 1: a_1t: c_1t + \text{h.o.t}]  
\end{align*}
It follows that $\lim_{t \to 0} \phi(\gamma_{1,a}(t)) = [0:0:1:0:0]$. It can be shown exactly in the same way that $\lim_{t \to 0} \phi(\gamma_{2,a}(t))$ is also $[0:0:1:0:0]$. Therefore, point $[0:0:1:0:0]$ is in the intersection of the closure in $X'$ of the curves $H_i(a) := \{x \in X: f_i(x) = a_i\}$ for $i = 1,2$. We conclude that completion $X'$ of $X$ does not preserve {\em any} fiber of $f$ at $\infty$.
\end{rem}

\begin{rem}
In the case that $n=2$, $X \subseteq \affine[m]{\kk}$ and $f:X \to \affine[2]{\kk}$ is a quasifinite map, we give an explicit construction of a projective completion of $X$ which preserves $f$ at $\infty$. It is the closure $X'$ of the graph $\Gamma$ of $f$ in $\pp^m(\kk)\times\pp^1(\kk)\times \pp^1(\kk)$.\\

To see this, identify $X$ with $\Gamma$ via $\phi:x \to (x,f(x))$. Let $a \in \kk$ be such that $H_{1,a} := \{x \in X: f_1(x) = a\} \neq \emptyset$. Then $\dim H_{1,a} = 1$. Let $C_1, \ldots, C_k$ be the irreducible components of $H_{1,a}$, $\bar C_i$ be the closure of $C_i$ in $X'$, and $\eta_i: \tilde C_i \to \bar C_i$ be the normalization map for each $i$. Since $f_2: C_i \to \kk$ is a non-constant map, it follows that $f_2 \circ \eta_i$ extends to a {\em finite} map from $\tilde C_i \to \pp^1(\kk)$. Let $B_i(a) := (f_2\circ\eta_i)(\tilde C_i\setminus \eta_i^{-1}(C_i))$ and let $\pi_3: X' \subseteq \pp^m(\kk)\times\pp^1(\kk)\times \pp^1(\kk)\to \pp^1(\kk)$ be the projection onto the $3$-rd coordinate. Pick an $x' \in \bar H_{1,a} \cap X_\infty$, where $\bar H_{1,a}$ is the closure of $H_{1,a}$ in $X'$. Then $x' \in \bar C_i\setminus C_i$ for some $i$. Pick a sequence of points $\{x_n\} \in C_i$ converging to $x'$. Also, pick $\tilde x \in \eta_i^{-1}(x')$ and $\tilde x_n \in \eta_i^{-1}(x_n)$ such that $\{\tilde x_n\}$ converge to $\tilde x$. Then $\pi_3(x') = \lim_{n\to \infty}\pi_3(\phi(x_n)) = \lim_{n\to \infty}f_2(x_n) = \lim_{n\to \infty}(f_2\circ\eta_i)(\tilde x_n) = (f_2\circ\eta_i)(\tilde x) \in B_i(a)$. Let $b \in \affine[1]{\kk}$. Then for every $x$ in $H_{2,b}:= \{x \in X: f_2(x) = b\}$, $\pi_3(x) = b$. This implies that $\pi_3(\bar H_{2,b}) = b$ and hence, if $b \not\in \cup_i B_i(a)$, then $\bar H_{1,a} \cap \bar H_{2,b} \cap X_\infty = \emptyset$. Since $B_i(a)$ is finite for each $a$, it follows that $X'$ preserves generic fibers of $f$ at $\infty$, as required. 
\end{rem}

\begin{example}\label{filtrexistence-thm2-example1}
Consider $f:\affine[2]{\kk} \to \affine[2]{\kk}$ given by $f(x,y) = (xy, x^2y-x+y)$. Then $f$ is quasi-finite, and $x$ and $y$ satisfy the following equations over $\kk[f_1,f_2]$:
\begin{eqnarray*}
x^2(f_1-1) - xf_2 + f_1 & = & 0, \\
y^2 - yf_2 + f_1^2 - f_1 & = & 0.
\end{eqnarray*}
In the notation of the proof of theorem \ref{filtrexistence-thm2} then
$$g = f_1 - 1$$
and the filtration $\scrF$ on $\kk[x,y]$ is given by $F_0 := \kk$, $F_1 := \kk\langle 1$, $x$, $y$, $f_1$, $f_2$, $f_1^2$, $xf_2$, $yf_2 \rangle$, $F_2 := F_1^2 + \kk\langle x^2f_1 \rangle$, and $F_k := \sum_{j=1}^{k-1} F_jF_{k-j}$ for $k>2$. By theorem \ref{filtrexistence-thm2}, for all $a = (a_1,a_2) \in \affine[2]{\kk}$ such that $a_1 \neq 1$, $\xf$ preserves $\finv(a)$ at $\infty$.
\end{example}

\section{\Gf and \Sgf} \label{sec-semi-quasidegree}
\subsection{Definitions and Examples} \label{subsec-semiquasintro}
\renewcommand{\filtrationring}{\ensuremath{A}}
\renewcommand{\filtrationchar}{\ensuremath{\mathcal{F}}}
\newcommand{\adelta}{\profingg{A}{\delta}}
\newcommand{\xdelta}{\profingg{X}{\delta}}

Below we relax the definition of degree like functions to include those having value $-\infty$. We will return to integer valued degree like functions after theorem \ref{quasi-integer-degree}.

\begin{defn*}
A {\em degree like function} on $A$ is a map $\delta: A\setminus \{0\} \to \zz \cup \{-\infty\}$ such that:
\begin{enumerate}
\item \label{degreelike-prop-0} $\delta(\kk) \leq 0$,
\item \label{degreelike-prop-1} $\delta(f+g) \leq \max\{\delta(f), \delta(g)\}$ for all $f, g \in A$, with $<$ in the preceding equation implying $\delta(f) = \delta(g)$.
\item \label{degreelike-prop-2} $\delta(fg) \leq \delta(f) + \delta(g)$ for all $f, g \in A$.
\end{enumerate}
\end{defn*}

There is a one-to-one correspondence between filtrations and degree like functions on $A$:\\
\begin{tabular}{ccc}
{\em Filtrations} & $\longleftrightarrow$ & {\em Degree like functions} \\
$\scrF	= \{F_d\}$			& $\longrightarrow$			& $\delta_\scrF:f \in A \to \inf\{d:f \in F_d\}$ \\
$\scrF_\delta := \{F_d := \{f \in A: \delta(f) \leq d\}\}$ & $\longleftarrow$ & $\delta$
\end{tabular}\\

Therefore, from now on we relate the degree like functions and filtrations by means of the natural
one-to-one correspondence described above. In particular, we call a degree like function $\delta$ {\em complete} (resp. {\em finitely generated}) iff the corresponding filtration $\scrF_\delta$ is complete (resp. finitely generated). Moreover, $\adelta$ and $\gr \adelta$ will denote the rings $\profingg{A}{\scrF_\delta}$ and $\gr \profingg{A}{\scrF_\delta}$ respectively, and $\psi_\delta$ will denote the natural mapping $\spec A \into \xdelta := \proj \adelta$.\\

We now introduce two classes of degree like functions which satisfy stronger versions of the multiplicative property (i.e. property \ref{degreelike-prop-2} above).

\begin{defn*}
\mbox{}
\begin{itemize}
\item A degree like function $\delta$ on $A$ is a {\em \gengf} iff $\delta(fg) = \delta(f) + \delta(g)$ for all $f,g \in A$.
\item We say that $\delta$ is a {\em \gensgff} if there are \gengff s $\delta_1, \ldots, \delta_N$ such that 
\begin{align} \label{gensgf-condition}
	\delta(f) = \max_{1 \leq i \leq N} \delta_{i}(f) \quad \text{for all}\ f \in A.
\end{align}
\end{itemize}
Given a \gensgf $\delta$ as in \ref{gensgf-condition}, getting rid of some $\delta_i$'s if necessary, we may assume that each $\delta_i$ that appears in \ref{gensgf-condition} is {\em not redundant}, i.e. for every $i$, there is an $f \in A$ such that $\delta_{i}(f) > \delta_{j}(f)$ for all $j \neq i$ (because, if there is $i$ such that for all $f \in A$, $\delta_i(f) \leq \delta_j(f)$ for some $j \neq i$, then $\delta(f) = \max_{j \neq i} \delta_{j}(f)$ for all $f \in A$). In the latter case we will say that \ref{gensgf-condition} is a {\em minimal presentation} of $\delta$.  
\end{defn*}

\newcommand{\akdeltaprime}{\profingg{A}{k\delta'}}

\begin{example}[Weighted Degree] \label{weighted-semi-example}
Weighted degrees on the polynomial ring $\kk[x_1, \ldots, x_n]$ are semidegrees. When the weight $d_i$ of $x_i$ is positive for each $i$, the weighted degree is complete and the associated completion is the corresponding weighted projective space.
\end{example}

\begin{example}[{\em Iterated} Semidegree]\label{iterated-semi-example}
Let $\delta$ be a degree like function on $A$. Pick $f \in A$ and an integer $w$ with $w < \delta(f)$. Consider $R := A[s]$, where $s$ is an indeterminate. Extend $\delta$ to a degree like function $\delta_e$ on $R$ by defining $\delta_e(\sum a_is^i) := \max_{a_i \neq 0}(\delta(a_i)+iw)$. 

\begin{claim*}
If $\delta$ is a semidegree then $\delta_e$ is also a semidegree.
\end{claim*}

\begin{proof}
Let $G = \sum g_is^i, H = \sum h_js^j, d := \delta_e(G), e := \delta_e(H)$. For each $k, l \geq 0$, write $G_k := \sum_{\delta(g_i)+ iw = k} g_is^i$ and $H_k := \sum_{\delta(h_i)+ iw = k} h_is^i$. It suffices to show $\delta_e(G_dH_e) \geq d + e$. Thus we may assume $G = G_d$ and $H = H_e$. Now let $i_0$ (resp. $j_0$) be the largest integer such that $g_{i_0} \neq 0$ (resp. $h_{j_0} \neq 0$). Then $GH = g_{i_0}h_{j_0}s^{i_0+j_0} + \sum_{m<i_0+j_0} a_ms^m$. Thus $\delta_e(GH) \geq \delta_e(g_{i_0}h_{j_0}s^{i_0+j_0}) = \delta(g_{i_0}h_{j_0})+ (i_0+j_0)w = \delta(g_{i_0})+ i_0w + \delta(h_{j_0})+ j_0w = d+e$.
\end{proof}

Let $J$ denote the ideal generated by $ s - f$ in $R$. Let $\tilde\delta$ be the degree like function on $A = R/J$ induced by $\delta_e$, i.e. $\tilde \delta(h) := \min\{\delta_e(H): H - h \in J\}$. Let $\gr \tilde I$ be the ideal generated by the class of $(f)_{\delta(f)}$ in $\gr \af$, and let $\tilde I$ be the inverse image of $\gr \tilde I$ under the natural projection $\af \onto \gr\af$. Since $\delta$ is a semidegree, it follows that $\tilde I = \langle (f)_{\delta(f)}\rangle$ and $\gr \tilde I = (\tilde I + \langle (1)_1 \rangle)/\langle (1)_1 \rangle = \langle (f)_{\delta(f)}, (1)_1 \rangle/\langle (1)_1 \rangle$.

\begin{thm}\label{iterated-thm} 
If $\delta$ is a semidegree, then $\tilde\delta$ is a semidegree if and only if $\gr \tilde I$ is a prime ideal of $\gring$.
\end{thm}

\begin{proof}
Consider maps $\phi: R \onto A$ and $\Phi: \profingg{R}{\delta_e} \onto \profinggg{\tilde \delta}$ such that $\phi(\sum a_is^i) := \sum a_if^i$, and $\Phi((H)_{d}) := (\phi(H))_d$. Then $\Phi$ is a surjective map of graded rings with $\ker \Phi = J^{\delta_e} := \psi_{\delta_e}(J) \subseteq \profingg{R}{\delta_e}$, so that $\profinggg{\tilde \delta} = \profingg{R}{\delta_e}/J^{\delta_e}$. It follows that $\gr \profinggg{\tilde \delta} = \profinggg{\tilde \delta}/\langle (1)_1 \rangle = \profingg{R}{\delta_e}/(J^{\delta_e} + \langle (1)_1 \rangle)$. Now $\profingg{R}{\delta_e} \cong \adelta[s]$ via the map $(\sum f_is^i)_d \to \sum (f_i)_{d-iw}s^i$. Moreover, since $\delta$ is a semidegree on $A$, it follows that $\delta_e$ is a semidegree on $R$. Therefore $J^{\delta_e} = \langle (s - f)_{\delta(f)} \rangle$, and hence $J^{\delta_e} + \langle (1)_1 \rangle = \langle (f)_{\delta(f)}, (1)_1 \rangle$. Thus $\gr \profinggg{\tilde \delta} = \adelta[s]/\langle (f)_{\delta(f)}, (1)_1 \rangle = (\adelta[s]/\langle (1)_1 \rangle)/ (\langle (f)_{\delta(f)}, (1)_1 \rangle/\langle (1)_1 \rangle) = (\gr \adelta/\gr \tilde I)[s]$. It follows that $\gr \profinggg{\tilde \delta}$ is an integral domain iff $\gr \tilde I$ is a prime ideal of $\gr \adelta$.
\end{proof}

In view of theorem \ref{iterated-thm}, we make the following definitions:
\begin{defn*}
Let $\delta$ be a semidegree on $A$.
\begin{itemize}
\item  The {\em leading term} $\ld(f)$ of an element $f$ of $A$ is the class of $(f)_{\delta(f)}$ in $\gr\adelta$. 
\item $f$ is  {\em prime} with respect to $\delta$ if the ideal $\langle \ld(f) \rangle$ is prime in $\gr \adelta$.
\end{itemize}  
\end{defn*}

\begin{example}\label{iterated-semi-concrete-example}
Let $A = \kk[x_1,x_2]$ and $\delta$ be the weighted degree on $A$ that assigns weight $3$ to $x_1$ and $2$ to $x_2$. Apply the procedure of example \ref{iterated-semi-example} to $f := x_1^2 - x_2^3 \in A$ and $w := 1 < \delta(f) = 6$ to define a new degree like function $\tilde \delta$ on $A$. Note that $\gr \adelta \cong \kk[x_1, x_2]$ via the map that sends $\ld(h)$ to the leading weighted homogenous component of $h$. Since $f$ is weighted homogenous, $\ld(f)$ corresponds to $f$ via the above identification. Since $f$ is irreducible in $\kk[x_1,x_2]$, it follows that $f$ is prime with respect to $\delta$, and by theorem \ref{iterated-thm}, $\tilde \delta$ is a semidegree. It is easy to see that $\adelta = \kk[(1)_1, (x_1)_3, (x_2)_2, (f)_1]$, and $\xdelta \cong V(ZW^5 - X^2 +Y^3) \subseteq \WP$, where $\WP$ is the weighted projective space $\pp^3(\kk;1,3,2,1)$ with (weighted homogenous) coordinates $[W:X:Y:Z]$.
\end{example}
\end{example}

\begin{example}[Quasidegrees determined by integral Polytopes]\label{toric-example}
Let $X$ be the $n$-torus $(\kk^*)^n$ and $A := \kk[x_1, x_1^{-1}, \ldots, x_n, x_n^{-1}]$ be its coordinate ring. Let $\scrP$ be a convex rational polytope (i.e. a convex polytope in $\rr^n$ with vertices in $\qq^n$) of dimension $n$ containing origin in the interior. Define $\delta' : A \setminus\{0\}\to \zz_+$ as follows:
\begin{align*}
\delta'(x^\alpha) &:= \inf\{r \in \qq_+: \alpha \in r\scrP\}, \\
\delta'(\sum a_\alpha x^\alpha) &:= \max_{a_\alpha \neq 0} \delta'(x^\alpha)
\end{align*}

\begin{claim*}
There is $k \in \nn$ such that $k\delta'$ is a complete quasidegree.
\end{claim*}

\begin{proof}  
For each face $\scrQ$ of $\scrP$, let $\eta_\scrQ$ be the smallest `outward pointing' integral vector normal to $\scrQ$ and let $c_\scrQ = \langle \eta_\scrQ, \alpha \rangle$, where $\alpha$ is any element of the hyperplane that contains $\scrQ$. Since $\scrP$ is rational, there is an $\alpha$ in each $\scrQ$ with rational coordinates, and hence each $c_\scrQ$ is a positive rational number. Let $\delta'_\scrQ$ be the $\qq$-valued weighted degree on $A$ given by:
$$\delta'_\scrQ(\sum a_\alpha x^\alpha) = \max_{a_\alpha \neq 0} \frac{\langle \eta_\scrQ, \alpha \rangle}{c_\scrQ}.$$
For each $r \in \rr_+$, $r\scrP = \{\beta \in \rr^n: \langle \eta_\scrQ, \beta \rangle \leq rc_\scrQ$ for every face $\scrQ$ of $\scrP\}$. It follows that for each $\alpha \in \rr^n$,
\begin{align*}
\delta'(x^\alpha) &:= \inf\{r \in \qq_+: \alpha \in r\scrP\}, \\
								 &= \inf\{r \in \qq_+: \langle \eta_\scrQ, \alpha \rangle \leq rc_\scrQ\ \text{for every face}\ \scrQ\ \text{of}\ \scrP\}\\
								 &= \inf\{r \in \qq_+: \delta'_\scrQ(x^\alpha) \leq r\ \text{for every face}\ \scrQ\ \text{of}\ \scrP\}\\
								 &= \max\{\delta'_\scrQ(x^\alpha): \scrQ\ \text{is a face of}\ \scrP\}.
\end{align*}
Let $k \in \nn$ be such that $k/c_\scrQ$ is an integer for each $\scrQ$. Then $k\delta'_\scrQ$ is an integer valued semidegree for each $\scrQ$ and hence $k\delta'$ is a quasidegree.\\

We claim that $\akdeltaprime$ is finitely generated. Indeed, since $\scrP$ is rational, there is an $l \in \nn$ such that $l\scrP$ is {\em integral}. Identify $\rr^n$ with the hyperplane $x_{n+1}=1$ in $\rr^{n+1}$. Let $\scrC$ be the cone in $\rr^{n+1}$ over $l\scrP \subseteq \rr^n$ (vertex of $\scrC$ being the origin). Then, with $d=kl$, the $d$-th truncated ring of $\akdeltaprime$ is
\begin{align*}
(\akdeltaprime)^{[kl]} &:= \dsum_{m \geq 0} \{f \in A: k\delta'(f) \leq klm\} \\
								&= \dsum_{m \geq 0} \{f \in A: \delta'(f) \leq lm\} \\
								&= \dsum_{m \geq 0} \kk\text{-span}\{x^\alpha: \alpha \in lm\scrP\} \\
								&= \dsum_{m \geq 0} \kk\text{-span}\{x^\alpha: (\alpha,m) \in \scrC \cap \zz^{n+1}\}.
\end{align*}
By Gordan's lemma \cite[Proposition 1, Section 1.2]{fultoric} the semigroup $\scrC \cap \zz^{n+1}$ is finitely generated. It follows that $(\akdeltaprime)^{[kl]}$ is a finitely generated $\kk$-algebra. But $\akdeltaprime$ is integral over $(\akdeltaprime)^{[kl]}$, since for every $(f)_d \in  \akdeltaprime$, $((f)_d)^{kl} \in (\akdeltaprime)^{[kl]}$. Therefore $\akdeltaprime$ is a $(\akdeltaprime)^{[kl]}$-submodule of the integral closure of $(\akdeltaprime)^{[kl]}$. Since the latter is a finite $(\akdeltaprime)^{[kl]}$-module, it follows that $\akdeltaprime$ is also a finite $(\akdeltaprime)^{[kl]}$-module and hence is a finitely generated algebra over $\kk$.
\end{proof}

Let $\delta := k\delta'$, where $k$ is an integer as in the above claim. Assume that $d$ is a positive integer such that $(\adelta)^{[d]}$ is generated as a $\kk$-algebra by elements in the $d$-th graded component of $\adelta$. These elements are precisely the $\kk$-span of monomials $x^\alpha$ such that $\alpha \in \frac{d}{k}\scrP$. The image of the $d$-uple embedding of $\xdelta$ is thus the closure (in the appropriate dimensional projective space) of the image of $(\kk^*)^n$ under the map induced by all the monomials with exponents in $\frac{d}{k}\scrP$. Hence $\xdelta$ is isomorphic via $d$-uple embedding to the classical {\em toric completion} $X_\scrP$ of $(\kk^*)^n$ determined by $\scrP$ \cite[secion 3.4]{fultoric}.
\end{example}
\subsection{Properties of \Sgf}\label{subsec-semiquasidegree-properties}
\renewcommand{\filtrationring}{\ensuremath{A}}
\renewcommand{\filtrationchar}{\ensuremath{\delta}}
\newcommand{\abarf}{\ensuremath{A^{\bar\scrF}}}
\newcommand{\afd}{{\af}^{(d)}}
\newcommand{\xfd}{{\xf}^{(d)}}
\newcommand{\del}[1]{(#1)_{\delta(#1)}}

\begin{thm}\label{gengfsgf-characterization}
Let $\delta$ be a degree like function on $A$ and let $I$ be the ideal of $\adelta$ generated by $(1)_1$. Then
\begin{enumerate}
\item\label{gengf-characterization} $\delta$ is a \gengf if and only if $I$ is a prime ideal, and
\item\label{gensgf-characterization} $\delta$ is a \gensgf if and only if $I$ is a {\em decomposable} radical ideal. In particular, if $\profing$ is noetherian, then $\delta$ is a \gensgf if and only if $I$ is a radical ideal.
\end{enumerate}
\end{thm}

\begin{proof}
$\delta$ is a \gengf iff for every $d, e \geq 1$, for every $f,g \in A$ with $\delta(f) = d$ and $\delta(g) = e$, $\delta(fg) = d+e$, i.e. iff for every $(f)_d, (g)_e \not\in I$, we have $(fg)_{d+e} \not\in I$, hence iff $I$ is a prime ideal. Therefore \ref{gengf-characterization} is proved.\\

Now assume $\delta$ is a \gensgff. We will show that $I$ is a decomposable radical ideal. Let $\delta = \max_{i=1}^N\delta_i$ be a minimal presentation for $\delta$. Then for each $i$ with $0 \leq i \leq N$, there is an $f_i$ such that $\delta_i(f_i) > \delta_j(f_i)$ for all $j \neq i$. In particular $\delta_i(f_i) \in \zz$. Let $d_i := \delta(f_i) = \delta_i(f_i)$. Recall the notation that $(I:\tilde f)$, where $\tilde f$ is an element of $\profing$, denotes the ideal of $\profing$ defined by: $(I: \tilde f) := \{\tilde g \in \profing:\ \tilde f \tilde g \in I\}$.

\begin{prolemma}\label{gensgf-characterize-lemma}
\mbox{}
\begin{enumerate}
\item \label{part1-gensgf-characterize-lemma} For each $i$, $1 \leq i \leq N$, $(I:(f_i)_{d_i})$ is a homogenous ideal, and the homogenous elements of $(I:(f_i)_{d_i})$ are precisely the elements of the set $L_i := \{(f)_{d}: f \in A, d > \delta_i(f)\}.$
\item \label{part2-gensgf-characterize-lemma} For each $i$, $(I:(f_i)_{d_i})$ is a distinct minimal prime ideal of $\profing$ containing $I$. Moreover,
\begin{align}\label{radical-decomposition}
\bigcap_{i=1}^N (I: (f_i)_{d_i}) &= I.
\end{align}
\end{enumerate}
\end{prolemma}

\begin{proof}
Fix an $i$, $1 \leq i \leq N$. Since $(f_i)_{d_i}$ is a homogenous element in $\profing$ and $I$ is a homogenous ideal of $\profing$, it follows that $(I: (f_i)_{d_i})$ is also a homogenous ideal of $\profing$. Let $(f)_d$ be an arbitrary homogenous element of $(I:(f_i)_{d_i})$. If $\delta(f) < d$, then $\delta_i(f) \leq \delta(f) < d$ and $(f)_d \in L_i$. So assume $\delta(f) = d$. Since $(f)_d(f_i)_{d_i} = (ff_i)_{d+d_i} \in I$, it follows that $\delta(ff_i) < d + d_i$. But then $\delta_i(ff_i) = \delta_i(f)+\delta_i(f_i) < d + d_i$, which implies that $\delta_i(f) < d = \delta(f)$, and thus $(f)_d \in L_i$. Therefore all homogenous elements of $(I:(f_i)_{d_i})$ belong to $L_i$. Now let $(f)_d$ be an arbitrary element of $L_i$. If $d > \delta(f)$, then $(f)_d \in I \subseteq (I:(f_i)_{d_i})$. So assume $d = \delta(f)$. Then $\delta_i(f) < d = \delta(f)$, and thus $\delta_i(ff_i) = \delta_i(f) + \delta_i(f_i) < d + d_i$. Also, for each $j \neq i$, $\delta_j(f_i) < d_i$, so that $\delta_j(ff_i) = \delta_j(f) + \delta_j(f_i) < d + d_i$. Therefore $\delta(ff_i) < d + d_i$ and $(f)_d(f_i)_{d_i} = (ff_i)_{d+d_i} \in I$. Thus $(f)_d \in (I:(f_i)_{d_i})$, which proves that $L_i \subseteq (I:(f_i)_{d_i})$, and hence part \ref{part1-gensgf-characterize-lemma} of the lemma is proved.\\

By definition of $f_i$ and $L_i$'s, we see that for each $i$, $(f_i)_{d_i} \in (\bigcap_{j \neq i}L_j) \setminus L_i$, which implies, by part \ref{part1-gensgf-characterize-lemma} of the lemma, that $(f_i)_{d_i} \in (\bigcap_{j \neq i}(I:(f_j)_{d_j})) \setminus (I: (f_i)_{d_i})$. Therefore $(I: (f_i)_{d_i})$ is a distinct ideal for each $i$. \\

Since $(I: (f_i)_{d_i})$ is a homogenous ideal, to prove that $(I: (f_i)_{d_i})$ is prime, all we need to show is that for all homogenous $(g_1)_{e_1},(g_2)_{e_2} \not\in I$ such that $(g_1)_{e_1}(g_2)_{e_2} \in (I: (f_i)_{d_i})$, one of the $(g_j)_{e_j}$'s also should lie in $(I: (f_i)_{d_i})$. So pick such $(g_1)_{e_1}, (g_2)_{e_2}$. Then $(g_1)_{e_1}(g_2)_{e_2} = (g_1g_2)_{e_1+e_2} \in (I: (f_i)_{d_i})$, which implies, by part \ref{part1-gensgf-characterize-lemma} of the lemma, that $\delta_i(g_1g_2) < e_1 + e_2$. Now, by our assumption, for each $j$, $(g_j)_{e_j} \not \in I$, so that $\delta(g_j) = e_j$, and hence $\delta_i(g_j) \leq e_j$. Taken together, the previous two statements imply that there is a $j$ such that $\delta_i(g_j) < e_j$. But then, by the part \ref{part1-gensgf-characterize-lemma} of the lemma again, $(g_j)_{e_j} \in (I: (f_i)_{d_i})$. It follows that $(I: (f_i)_{d_i})$ is prime.\\

If $I \subseteq p \subseteq (I: (f_i)_{d_i})$ for a prime ideal $p$ of $\profing$, and some $i \leq N$, then $(f_i)_{d_i} \not\in p$ (since $(f_i)_{d_i} \not\in (I: (f_i)_{d_i})$ by part \ref{part1-gensgf-characterize-lemma}). But then for each $(g)_e \in \profing$ such that $(g)_e(f_i)_{d_i} \in I \subseteq p$, $(g)_e \in p$, so that $(I: (f_i)_{d_i}) \subseteq p$, and hence $(I: (f_i)_{d_i}) = p$. Thus for each $i$, $(I: (f_i)_{d_i})$ is a minimal prime ideal containing $I$.\\

For the last part of the lemma, pick any homogenous $(g)_e \in \bigcap (I: (f_i)_{d_i})$. Then, by part \ref{part1-gensgf-characterize-lemma}, for each $i$, $\delta_i(g) < e$. Therefore $\delta(g) = \max_i \delta_i(g)< e$, and hence $(g)_e \in I$. Since $\bigcap(I: (f_i)_{d_i})$ is a homogenous ideal containing $I$, and $(g)_e$ was an arbitrary homogenous element in that ideal, it follows that $\bigcap (I: (f_i)_{d_i}) = I$.
\end{proof}

The second statement of the above lemma tells, in the language of \cite[Chapter 4]{am} that $I$ is a {\em decomposable ideal} in $\profing$, \ref{radical-decomposition} is the {\em unique} minimal primary decomposition of $I$, and $(I: (f_i)_{d_i})$ are the {\em minimal} prime ideals belonging to $I$. Also, since $I$ is a finite intersection of prime ideals, it follows that $I$ is radical. Hence we have proved one half of the second statement of the proposition.\\

Now assume $I$ is a decomposable radical ideal of $\profing$. We will have to show that $\delta$ is a \gensgff. By the theory of minimal decomposition of ideals \cite[Chapter 4]{am} there exist $(f_1)_{d_1}, \ldots, (f_N)_{d_N} \in \profing \setminus I$ such that $(I:(f_i)_{d_i})$ is a minimal prime ideal containing $I$ for each $i$, and the minimal primary decomposition of $I$ is given by:
\begin{align*}
I &= \bigcap_{i=1}^N (I:(f_i)_{d_i}),
\end{align*}
Since $(f_i)_{d_i} \not\in I$, it follows that $\delta(f_i) = d_i \in \zz$ for each $i$. For each $i=1, \ldots, N$ define $\delta_i:A \to \zz \cup \{-\infty\}$ as follows:
\begin{align} \label{defn-d_i-from-f_i}
\delta_i(f) &:= \lim_{k \to \infty} \delta((f_i)^kf) - \delta((f_i)^k).
\end{align}
Note that the above definition is bounded from above. Indeed, for each $i$, $(f_i)_{d_i} \not\in I$. But $I$ is a radical ideal, so that for each $k\geq 1$, $((f_i)_{d_i})^k = ((f_i)^k)_{kd_i} \not\in I$; which implies that $\delta((f_i)^k) = kd_i$. Thus for each $k \geq 1$,
\begin{align*}
\delta((f_i)^{k+1}f) - \delta((f_i)^{k+1}) &\leq \delta((f_i)^kf) + \delta(f_i) - \delta((f_i)^{k+1})\\
		     			   &= \delta((f_i)^kf) + d_i - (k+1)d_i \\
		     			   &= \delta((f_i)^kf) - kd_i \\
		     			   &= \delta((f_i)^kf) - \delta((f_i)^k),
\end{align*}
and hence $\delta_i(f)$ is a well defined element in $\zz \cup \{-\infty\}$. Now we prove that $\delta_i$ is multiplicative.\\

At first note that $\delta_i(f)\in \zz$ iff there is a $k_f \in \nn$ such that for all $k \geq k_f$, $\delta((f_i)^kf) = \delta((f_i)^k) + \delta_i(f) = kd_i + \delta_i(f)$, which is true iff for all $k \geq k_f$, $((f_i)^kf)_{kd_i+\delta_i(f)}(f_i)_{d_i} \not\in I$, i.e. iff for all $k \geq k_f$, $((f_i)^kf)_{kd_i+\delta_i(f)} \not\in (I:(f_i)_{d_i})$. Therefore, if both $\delta_i(f)$ and $\delta_i(g)$ are integers, then there is a $k' := \max\{k_f,k_g\} \in \nn$ such that for all $k \geq k'$, $((f_i)^kf)_{kd_i+\delta_i(f)}$ and $((f_i)^kg)_{kd_i+\delta_i(g)}$ do not lie in $(I:(f_i)_{d_i})$. Since $(I:(f_i)_{d_i})$ is a prime ideal, it follows that for all $k \geq k'$, $((f_i)^kf)_{kd_i+\delta_i(f)}((f_i)^kg)_{kd_i+\delta_i(g)} = ((f_i)^{2k}fg)_{2kd_i+\delta_i(f)+\delta_i(g)} \not\in (I:(f_i)_{d_i})$. It follows by the observation in the first sentence of this paragraph that $\delta_i(fg) = \delta_i(f)+\delta_i(g) \in \zz$.\\

Now, $\delta_i(f) = -\infty$ if and only if for each $n \in \zz$ there is a $k' \in \nn$ such that for all $k \geq k'$, $\delta((f_i)^kf)-\delta((f_i)^k) < n$. Pick $f,g \in A$ such that $\delta_i(f) = -\infty$. If $\delta(g)$ is $-\infty$, then for all $k \geq 0$, $\delta((f_i)^kfg) = -\infty$, which implies that $\delta_i(fg) = -\infty = \delta_i(f) + \delta_i(g)$. So assume $\delta(g) \in \zz$. Pick any $n \in \zz$. Let $k' \in \nn$ such that for all $k \geq k'$, $\delta((f_i)^kf)-\delta((f_i)^k) < n - \delta(g)$. Then for all $k \geq k'$, $\delta((f_i)^kfg)-\delta((f_i)^k) \leq \delta((f_i)^kf)+ \delta(g) -\delta((f_i)^k) < n$. It follows that $\delta(fg) = -\infty = \delta(f) + \delta(g)$. Thus $\delta_i$ is multiplicative for each $i$. Similarly it can be shown that each $\delta_i$ satisfies the additive property of degree like functions. Therefore each $\delta_i$ is a \gengff. We now show that $\delta = \max_{i=1}^N(\delta_i)$.\\

First of all note that for all $i$, $f \in A$, and all $k \in \nn$, the following inequality holds: $\delta((f_i)^kf) - \delta((f_i)^k) \leq \delta(f)$, which implies that $\delta_i(f) \leq \delta(f)$. Thus $\delta \geq\max_{i=1}^N (\delta_i)$. Now pick any $f \in A$. If $\delta(f) = -\infty$, then for all $i$ and for all $k \geq 0$, $\delta((f_i)^kf) = -\infty$, and hence $\delta_i(f) = -\infty = \delta(f)$. So assume $\delta(f) \in \zz$. Since $(f)_{\delta(f)} \not\in I$, it follows that there is an $i$ such that $(f)_{\delta(f)} \not\in (I:(f_i)_{d_i})$. Now, note that $(f_i)_{d_i} \not\in (I:(f_i)_{d_i})$, for that would imply $((f_i)_{d_i})^2 \in I$, which would in turn imply $(f_i)_{d_i} \in I$ (since $I$ is a radical ideal) contradicting our choice of $(f_i)_{d_i}$. Thus neither of $(f)_{\delta(f)}$ and $(f_i)_{d_i}$ is an element of $(I:(f_i)_{d_i})$. Since $(I:(f_i)_{d_i})$ is a prime ideal, it follows that for all $k \geq 0$, $((f_i)_{d_i})^k(f)_{\delta(f)} = ((f_i)^kf)_{kd_i+\delta(f)} \not\in (I:(f_i)_{d_i})$. But then $\delta_i(f) = \delta(f)$. Thus $\delta =\max_{i=1}^N (\delta_i)$, and hence $\delta$ is indeed a \gensgff.
\end{proof}

\begin{cor} \label{sgf-unique-decomp}
Let $\delta$ be a \gensgf on $A$, then there exist unique \gengff s $\delta_1, \ldots, \delta_N$ such that $\delta = \max_{1\leq i\leq N} \delta_i$ is a minimal presentation of $\delta$.
\end{cor}

\begin{proof}
Let $\delta_1, \ldots, \delta_{N}, \delta'_1, \ldots, \delta'_{N'}$ be \gengff s such that $\delta = \max_{1\leq i\leq N} \delta_i = \max_{1\leq i\leq N'} \delta'_i$ are two minimal presentations of $\delta$. As in the proof of theorem \ref{gengfsgf-characterization}, there exist $f_1, \ldots, f_n,  f'_1, \ldots, f'_{N'} \in A$  such that
\begin{align}
	\arraycolsep 0em
	\begin{array}{rl}\label{f_i-f'_i-property}
	\delta_i(f_i) &> \delta_j(f_i)\ \textrm{for}\ 1\leq j \neq i \leq N,\ \textrm{and}\\
	\delta'_i(f'_i) &> \delta'_j(f'_i)\ \textrm{for}\ 1\leq j \neq i \leq N'.
	\end{array}
\end{align}

Then $d_i := \delta_i(f_i)$ and $d'_i := \delta'_i(f'_i)$ are integers, and by lemma \ref{gensgf-characterize-lemma},
$$I = \bigcap_{i=1}^{N} (I: (f_i)_{d_i}) = \bigcap_{i=1}^{N'} (I: (f'_i)_{d'_i})$$
are {\em unique} minimal primary decompositions of $I$, where $I$ is the ideal in $\profing$ generated by $(1)_1$. But then $N=N'$, and after a re-indexing of $\delta'_i$'s if necessary, the ideals $(I:(f'_i)_{d'_i})$ and $(I:(f'_i)_{d'_i})$ coincide for each $i \leq N$. \\

Fix an $i$, $1 \leq i \leq N$. By the proof of lemma \ref{gensgf-characterize-lemma}, $(f'_i)_{d'_i} \in (\bigcap_{j \neq i}(I:(f'_j)_{d'_j})) \setminus (I: (f'_i)_{d'_i}) = (\bigcap_{j \neq i}(I:(f_j)_{d_j})) \setminus (I: (f_i)_{d_i})$. Thus, by part \ref{part1-gensgf-characterize-lemma} of lemma \ref{gensgf-characterize-lemma}, $(f'_i)_{d'_i} \in (\bigcap_{j \neq i}L_j) \setminus L_i$, where $L_j$'s are defined as in lemma \ref{gensgf-characterize-lemma}. But then, by definition of $L_j$'s, $\delta_i(f'_i) = \delta(f'_i) = \delta'_i(f'_i)$, and $\delta_j(f'_i) < \delta_i(f'_i)$ for all $j \neq i$. Since this is true for each $i$, it follows that the property \ref{f_i-f'_i-property} holds even if we substitute each $f_i$ by $f'_i$. Since $f_i$'s were assumed to be arbitrary elements in $A$ such that \ref{f_i-f'_i-property} is true, we can without any trouble assume that
$f_i = f'_i$ for each $i$. Finally, note that by the proof of theorem \ref{gengfsgf-characterization}, the following holds for all $f \in A$:
\begin{align*}
\delta_i(f) &= \lim_{k \to \infty} \delta((f_i)^kf) - \delta((f_i)^k),\ \textrm{and} \\
\delta'_i(f) &= \lim_{k \to \infty} \delta((f'_i)^kf) - \delta((f'_i)^k)
\end{align*}
Since $f'_i = f_i$, it follows that $\delta'_i = \delta_i$ for each $i$.
\end{proof}

\begin{cor}
Let $\delta$ be a {\em \gensgff} on $A$ with minimal presentation $\delta = \max_{i=1}^N \delta_i$. Then $X_\infty:=\proj \profing \setminus \spec A$ has $N$ irreducible components.
\end{cor}

\begin{proof}
By theorem \ref{completethm}, $X_\infty \cong \proj \gring$, where $\gring := \dsum_{i\geq 0}(F_i/F_{i-1}) \cong \profing/I$, where $I$ is the ideal in $\profing$ generated by $(1)_1$. Let $f_1, \ldots, f_N$, $d_1, \ldots, d_N$ be as in the proof of theorem \ref{gengfsgf-characterization} and let $\mu_i$ be the class of $(f_i)_{d_i}$ in $\gring$. It follows by lemma \ref{gensgf-characterize-lemma} that $0 = \bigcap_{i=1}^N (0:\mu_i)$ is the minimal primary decomposition of the zero ideal in $\gring$ and $(0:\mu_i)$'s are the unique minimal prime ideals belonging to the zero ideal. Since $(0:\mu_i)$ is also homogenous for each $i$, it follows that $X_\infty = \proj \gring = \bigcup_{i=1}^N V((0:\mu_i))$ is the decomposition of $X_\infty$ into irreducible components.
\end{proof}

If $A$ is an integrally closed domain, we deduce additional properties of $\adelta$.

\begin{prop}\label{quasi-normal-prop}
If $A$ is an integrally closed domain and $\delta$ is a \gensgf on $A$, then $\profing$ is also integrally closed.
\end{prop}

\begin{proof}
Consider the isomorphism $\profing \cong \dsum_{i \in \zz} F_it^i \subseteq A[t,t^{-1}]$, where $(1)_1$ gets mapped to $t$. Since $A$ is an integrally closed domain, it follows that $A[t,t^{-1}]$ is also integrally closed \cite[Exercise 5.9]{am}. So it suffices to show that $\profing$ is integrally closed in $A[t,t^{-1}]$. Pick $f = \sum_{i=q}^r f_it^i \in A[t,t^{-1}]$ integral over $\profing$, where $f_i \in A$ for each $i$. Then $f$ satisfies an eqation of the form
\begin{align*}
X^s + G_1X^{s-1}+ \cdots + G_s &= 0,
\end{align*}
for $G_1, \ldots, G_s \in \profing$. Taking the highest degree terms in $t$, we see that $f_rt^r$ is integral over $f$. Substituting $f$ by $f-f_rt^r$ and repeating the procedure, we see that each $f_it^i$ is integral over $\profing$. Thus it suffices to show that if $ft^k$ is integral over $\profing$ for some $f \in A$, then $ft^k \in \profing$, or equivalently, $\delta(f) \leq k$. So take $f \in A$ such that $ft^k$ satisfies an equation of the above form. Taking coefficients of $t^{ks}$ if necessary, we can assume that $G_i = g_it^{ik}$ for some $g_i \in A$. Since $g_it^{ik} \in \profing$, it follows that $\delta(g_i)\leq ik$. Assume $d := \delta(f) > k$. Plugging in $t=1$ into the above equation, we see that
\begin{align*}
f^s = -\sum_{i=1}^s g_if^{s-i}
\end{align*}
in $A$. For each $i \geq 1$, we have, $\delta(g_if^{s-i}) \leq \delta(g_i) + \delta(f^{s-i}) \leq ik + (s-i)d < id + (s-i)d = sd$. It follows that $\delta(f^s) = \delta(-\sum_{i=1}^s g_if^{s-i}) \leq \max_{i=1}^s \delta(g_if^{s-i}) < sd$. But since $\delta$ is a \gensgff, it follows that $\delta(f^s) = s\delta(f) = sd$, a contradiction! Thus $d \leq k$.
\end{proof}

\begin{cor}
Let $X$ be a normal affine variety with $A$ being the ring of regular functions on $X$ and $\delta$ being a \complete \sgf on $A$. Then $\spec \profing$ and $\proj \profing$ are also normal varieties.
\qed
\end{cor}

\begin{example}
The converse of proposition \ref{quasi-normal-prop} is false. Consider the degree like function $\delta$ on $\kk[x]$ defined by: 
\begin{align*}
\delta(x^k) := \left\{
									\begin{array}{cl}
									3k/2 & \text{if}\ k\ \text{is even} \\
									3(k-1)/2+2 & \text{if}\ k\ \text{is odd} 
									\end{array}
								\right.
\end{align*}
Then $\profingg{\kk[x]}{\delta} \cong \kk[x,y,z]/\langle x^2 - yz \rangle$ (where $\kk[x,y,z]$ is graded by weights 2, 3, 1 corresponding respectively to $x, y, z$). If $\kk$ is not characteristic $2$, then $\profingg{\kk[x]}{\delta} $ is integrally closed \cite[Exercise II.6.4]{hart}. On the other hand, $\delta(x^2) = 3 < 4 = 2\delta(x)$, so that $((x)_2)^2 \in I$ even though $(x)_2 \not\in I$, where $I$ is the ideal of $\profingg{\kk[x]}{\delta}$ generated by $(1)_1$. Thus $I$ is not radical. It follows by theorem \ref{gengfsgf-characterization} that $\delta$ is not a \sgff.
\end{example}

The following theorem gives a characterization of the semidegrees $\delta_i$ associated to a quasidegree $\delta$, provided that $\adelta$ is finitely generated. In particular, it states that if $\delta$ is integer valued, then each $\delta_i$ is also integer valued (which is not clear from the limit definition of $\delta_i$'s from the proof of theorem \ref{gengfsgf-characterization}). We use the notion of a {\em Krull domain} in the theorem, so we recall its definition: 

\begin{defn*}
An integral domain $B$ is a {\em Krull domain} iff 
\begin{enumerate}
\item $B_\ppp$ is a discrete valuation ring for all height one prime ideals $\ppp$ of $B$, and
\item every non-zero principal ideal of $B$ is the intersection of a finite number of primary ideals of height one.
\end{enumerate}
\end{defn*}

Every normal noetherian domain is a Krull domain \cite[Section 41]{matsulgebra}. In particular, the integral closure of $\adelta$ is a Krull domain provided that $\adelta$ is finitely generated.

\begin{thm}\label{quasi-integer-degree}
Let $\delta$ be a quasidegree on $A$ with a minimal presentation $\delta = \max_{1 \leq i \leq N}\delta_i$. Assume that $\adelta$ is a finitely generated $\kk$-algebra. Let $B$ be any Krull domain which is also an integral extension of $\adelta$ and $\ppp_1, \ldots, \ppp_r$ be the height one primes of $B$ containing $(1)_1$. For each $j$, $1 \leq j \leq r$, define a function $\hat\delta_j$ on $A\setminus\{0\}$ by
\begin{align*}
\hat\delta_j(f) := \left\{
											\begin{array}{cl}
											-\infty & \text{if}\ \delta(f) = -\infty \\ 
											\delta(f) - \frac{\nu_j((f)_{\delta(f)})}{e_j} & \text{otherwise},
											\end{array}
										\right.
\end{align*}
where $\nu_j$ is the discrete valuation of the discrete valuation ring $B_{\ppp_j}$ and $e_j := \nu_j((1)_1)$. Then for each $i$, $1 \leq i \leq N$, semidegree $\delta_i \equiv \hat\delta_j$ for some $j$, $1 \leq j \leq r$. In particular, if $\delta$ is integer valued, then $\delta_i$ is also integer valued for each $i$. 
\end{thm}

\begin{proof}
Let $I$ be the ideal generated by $(1)_1$ in $\adelta$. By theorem \ref{gengfsgf-characterization}, $I$ is radical. It follows from \cite[Lemma 11.3 and 11.4]{mcdivisors} that for all $f \in A$ such that $\delta(f) \neq -\infty$, $\min_{j=1}^r \frac{\nu_j((f)_{\delta(f)})}{e_j} = 0$, where $e_j := \nu_j((1)_1)$ for each $j$. Therefore $\delta(f) = \delta(f) - \min_{j=1}^r \frac{\nu_j((f)_{\delta(f)})}{e_j} = \max_{j=1}^r \hat\delta_j(f)$ for all $f \in A$.

\begin{claim*} 
Each $\hat\delta_j$ satisfies the additive and multiplicative properties of a semidegree.
\end{claim*}

\begin{proof}
Fix a $j$, $1 \leq j \leq r$. Let $f, g \in A\setminus\{0\}$. If either $\delta(f)$ or $\delta(g)$ is $-\infty$, then it is straightforward to see that both $\hat\delta_j(fg)$ and $\hat\delta_j(f+g)$ satisfy the properties required of a semidegree. So assume both $\delta(f)$ and $\delta(g)$ are integers. At first we verify the multiplicative property. Let $\epsilon := \delta(f) + \delta(g) - \delta(fg)$. Then $\hat\delta_j(fg) = \delta(fg) - \frac{\nu_j((fg)_{\delta(fg)})}{e_j} = \delta(f) + \delta(g) - \frac{e_j\epsilon+\nu_j((fg)_{\delta(fg)})}{e_j} =  \delta(f) + \delta(g) - \frac{\nu_j((1)_\epsilon) + \nu_j((fg)_{\delta(fg)})}{e_j} =  \delta(f) + \delta(g) - \frac{\nu_j((fg)_{\delta(f)+\delta(g)})}{e_j} = \delta(f) + \delta(g) - \frac{\nu_j((f)_{\delta(f)}) + \nu_j((g)_{\delta(g)})}{e_j} = \hat\delta_j(f) + \hat\delta_j(g)$.\\

To verify the additive property, we first assume $d :=\delta(f) - \delta(g) > 0$, so that $\delta(f+g) = \delta(f)$ and $\del{f+g} = \del{f} + ((1)_1)^d\del{g}$. If $\nu_j(\del{f}) < de_j + \nu_j(\del{g})$, then $\nu_j(\del{f+g}) = \nu_j(\del{f})$ and hence $\hat\delta_j(f+g) = \hat\delta_j(f)$. Otherwise $\nu_j(\del{f+g}) \geq de_j + \nu_j(\del{g})$ and $\hat\delta_j(f+g) = \delta(f+g) - \frac{\nu_j(\del{f+g})}{e_j} \leq \delta(f) - d - \frac{\nu_j(\del{g})}{e_j} = \delta(g) - \frac{\nu_j(\del{g})}{e_j} = \hat\delta_j(g)$. Now assume $\delta(f) = \delta(g)$ and $\hat\delta_j(f) \geq \hat\delta_j(g)$. It follows that $\nu_j((f)_{\delta(f)}) \leq \nu_j((g)_{\delta(g)})$, and hence $\nu_j((f+g)_{\delta(f)}) = \nu_j((f)_{\delta(f)} + (g)_{\delta(g)}) \geq \nu_j((f)_{\delta(f)})$. Let $e := \delta(f) - \delta(f+g) \geq 0$. Then $\hat\delta_j(f+g) = \delta(f+g) - \frac{\nu_j(\del{f+g})}{e_j} = \delta(f) - e - \frac{\nu_j(\del{f+g})}{e_j} =  \delta(f) - \frac{\nu_j(((1)_1)^e\del{f+g})}{e_j} = \delta(f) - \frac{\nu_j((f+g)_{\delta(f)})}{e_j} \leq \delta(f) - \frac{\nu_j(\del{f})}{e_j} = \hat\delta_j(f)$. Therefore, $\hat\delta_j(f+g) \leq \max\{\hat\delta_j(f), \hat\delta_j(g)\}$. Similar arguments imply that if $\hat\delta_j(f) > \hat\delta_j(g)$, then $\hat\delta_j(f+g) = \hat\delta_j(f)$. 
\end{proof}
The first assertion of the theorem now follows from the uniqueness of the minimal presentation of a quasidegree. For the last statement, note that for all $f \in A\setminus\{0\}$, if $\delta(f) \in \zz$, then $\hat\delta_j(f) \in \qq$ for each $j$, $1 \leq j \leq r$, and hence $\delta_i(f) \neq -\infty$ for each $i$, $1 \leq i \leq N$.
\end{proof}

From this point on, all of our degree like functions will be integer valued. In the course of the proof of our main existence theorem to be stated below, we make use of the theory of {\em Rees' valuations} (see \cite[chapter XI]{mcdivisors}). For an ideal $I$ of a ring $R$ let $\nu_I: R \to \nn \cup \{\infty\}$ and $\bar \nu_I: R \to \qq_+ \cup \{\infty\}$ be defined by: $\nu_I(x) := \sup\{m: x \in I^m\}$ and $\bar \nu_I(x) := \lim_{m \to \infty} \frac{\nu_I(x^m)}{m}$. The following is due to Rees \cite[propositions 11.1, 11.5, corollary 11.6]{mcdivisors}:

\begin{thm*}[Rees]
For any ring $R$ and any ideal $I$ of $R$, $\bar \nu_I$ is well defined. If $R$ is a noetherian domain then
\begin{enumerate}
\item there is a positive integer $e$ such that for all $x \in R$, $\bar \nu_I(x) \in \frac{1}{e}\nn$, and
\item if $k \geq 0$ is an integer then $\bar \nu_I(x) \geq k$ if and only if $x \in \bar{I^k}$, where $\bar{I^k}$ is the integral closure of $I^k$ in $R$.
\end{enumerate}
\end{thm*}

\begin{defn*}
Let $\delta$ and $\eta$ be degree like functions on $A$.
\begin{itemize}
\item We say $\eta$ is {\em integral} over $\delta$ if $\eta(f) \leq \delta(f)$ for all $f \in A$, so that there is a natural inclusion $\profing \into \profinggg{\eta}$ and $\profinggg{\eta}$ is integral over $\profing$ under the above inclusion.
\item $\eta$ {\em preserves the intersections at $\infty$ for the completion determined by $\delta$} (in short, {\em for $\delta$}), if for any closed subsets $X_1, \ldots, X_k$ of $\spec A$ such that the closures of $X_i$ in $\proj \profing$ do not intersect at any point at infinity, the closures of $X_i$ in $\proj \profinggg{\eta}$ also do not intersect at any point at infinity.
\end{itemize}
\end{defn*}

\begin{lemma}\label{inf-intersection-preserve-lemma}
Let $\delta$ and $\eta$ be \nonnegative degree like functions on $A$. If $\eta$ is integral over $\delta$ or $\eta = n\delta$ for some $n > 0$, then $\eta$ preserves the intersections at $\infty$ for $\delta$.
\end{lemma}

\begin{proof}
For an ideal $I$ of $A$, we denote by $I^{\delta}$ (respectively $I^{\eta}$) the `homogenization' of $I$ in $\profing$ (respectively $\profinggg{\eta}$). Recall that $\profing_+$ (resp. $\profinggg{\eta}_+$) denotes the irrelevant homogenous ideal of $\profing$ (resp. $\profinggg{\eta}$). Let $X_j := V(I_j)$ be closed subsets of $\spec A$ for $1 \leq j \leq k$. Assume that the closures of $X_i$ in $\proj \profing$ do not intersect at any points at infinity. It follows that the radical of the ideal $\scrI_\delta := \langle (I_1)^\delta, \ldots, (I_k)^\delta, (1)_1 \rangle$ equals $\profing_+$.\\

At first assume $\eta$ is integral over $\delta$. Then $\eta \leq \delta$ and there is a natural inclusion $\profing \subseteq \profinggg{\eta}$. Moreover, under the above inclusion, $I^\delta \subseteq I^\eta$ for each ideal $I$ of $A$. It follows that $\scrI_\delta \subseteq \scrI_\eta$, where $\scrI_\eta := \langle (I_1)^\eta, \ldots, (I_k)^\eta, (1)_1 \rangle \subseteq \profinggg{\eta}$. But then $\profing_+ = \sqrt \scrI_\delta \subseteq \sqrt \scrI_\eta$. Since each element of $\profinggg{\eta}_+$ is integral over $\profing_+$, it follows that $\profinggg{\eta}_+ \subseteq \sqrt \scrI_\eta$, and therefore the closures of $X_i$ in $\proj \profinggg{\eta}$ do not intersect at any points at infinity.\\

Now assume $\eta = n\delta$. Let $\scrF := \{F_d\}_{d \geq 0}$ (resp. $\scrG := \{G_d\}_{d \geq 0}$) be the filtration corresponding to $\delta$ (resp. $\eta$). Recall that $\profing \cong \dsum_{d \geq 0}F_dt^d$ and $\profinggg{\eta} \cong \dsum_{d \geq 0}G_ds^d$, where $s,t$ are indeterminates over $A$. But for each $d \geq 0$, $G_d := \{f \in A: \eta(f) \leq d\} = \{f \in A: \delta(f) \leq d/n\}  = F_{\lfloor d/n \rfloor}$. Therefore $\profinggg{\eta} \cong \dsum_{d \geq 0} F_{\lfloor d/n \rfloor}s^d$. The map $t \mapsto s^n$ gives an inclusion of $\profing$ into $\profinggg{\eta}$ such that $\profinggg{\eta}$ is integral over $\profing$. The same argument as in the previous paragraph then gives the desired result.
\end{proof}

\begin{thm}[Main Existence Theorem]\label{thm-noeth-integral-quasidegree}
Let $\delta$ be a finitely generated degree like function on $A$ (in particular, $\adelta$ is finitely generated). Then there is a finitely generated filtration $\tilde\delta$ on $A$ and a positive integer $e$ such that $\tilde \delta$ is a quasidegree and $\tilde \delta$ is integral over $e\delta$. If $\delta$ is non-negative, then $\tilde \delta$ is also non-negative, and $\tilde \delta$ preserves the intersections at $\infty$ for $\delta$.
\end{thm}

\begin{proof}
Let $\scrF := \{F_d\}_{d \in \zz}$ be the filtration on $A$ corresponding to $\delta$. Consider the isomorphism $\adelta \cong \bigoplus_{i \in \zz} F_it^i$, where $t$ is an indeterminate and $(1)_1$ is mapped to $t$. Let $I$ be the ideal in $\adelta$ generated by $(1)_1$. With $\nu_I$ and $\bar \nu_I$ as defined above, let $e$ be a positive integer provided by Rees' theorem such that for all $H \in \af$, $\bar \nu_I(H) \in \frac{1}{e}\nn$.\\

Fix $h \in A$ and $m \in \nn$. Then $\delta(h^m) \leq m\delta(h)$. Moreover, since $I$ is generated by $(1)_1$, it follows that $k:=m\delta(h) - \delta(h^m)$ is the largest integer such that $(h^m)_{m\delta(h)} \in I^{k}$. The definition of $\nu_I$  implies that $\nu_I(((h)_{\delta(h)})^m) = k = m\delta(h) - \delta(h^m)$. Therefore $\delta(h^m)/m = \delta(h) - \nu_I(((h)_{\delta(h)})^m)/m$. It follows that $\bar\delta(h) := \lim_{m \to \infty} \delta(h^m)/m$ is well defined and equals $\delta(h) - \bar \nu_I((h)_{\delta(h)})$. \\

For $m \in \zz$, let $\bar F_{\frac{m}{e}} := \{f \in A: \bar\delta(h) \leq \frac{m}{e}\}$, and consider ring $\profinggg{\bar\delta} := \bigoplus_{m \in \zz} \bar F_{\frac{m}{e}}t^{\frac{m}{e}}$. Since $\bar \delta \leq \delta$, it follows that $F_k \subseteq \bar F_k$ for each $k \in \zz$. Therefore $\adelta \subseteq \profinggg{\bar\delta}$.

\begin{claim*}
$\profinggg{\bar\delta}$ is integral over $\adelta$.
\end{claim*}

\begin{proof}
It suffices to show that for each $h \in A$, $(h)_{\bar\delta(h)}$ is integral over $\adelta$. Pick $h \in A$. Then $(h)_{\bar\delta(h)}$ is integral over $\adelta$ if and only if $\bar H := ((h)_{\bar\delta(h)})^e$ is integral over $\adelta$. From construction of $\bar\delta$ it follows that $\bar H = (h^e)_{e\bar\delta(h)} = (h^e)_{\bar\delta(h^e)}$. Let $H := (h^e)_{\delta(h^e)} \in \adelta$ and $\bar k := \bar \nu_I(H)$. Then $\bar k =  \delta(h^e) - \bar\delta(h^e) = \delta(h^e) - e\bar\delta(h)$. It follows that $\bar k$ is an integer. Hence by Rees' theorem, $H$ is in the integral closure of $I^{\bar k}$ in $\adelta$, i.e. $H$ satisfies an equation of the form $H^l + G_1H^{l-1}+ \cdots + G_l = 0$, where $G_i \in I^{i\bar k}$ for each $i$. Since $\profinggg{\bar\delta}$ is a graded ring, we may assume without loss of generality that the degrees of $G_i$ are $i\delta(h^e)$. Then $G_i = (g_i)_{i(\delta(h^e)-\bar k)}(1)_{i\bar k}$ for some $g_i \in A$ with $\delta(g_i) \leq i(\delta(h^e)-\bar k) = i\bar\delta(h^e)$, $1 \leq i \leq l$. Moreover, in the ring $\profinggg{\bar\delta}$, $H = (h^e)_{\delta(h^e)} = (h^e)_{\bar\delta(h^e)+\bar k} = (h^e)_{\bar\delta(h^e)}(1)_{\bar k} = (1)_{\bar k}\bar H$. Substituting these values of $H$ and $G_i$ into the equation of integral dependence for $H$ and then cancelling a factor of $(1)_{l\bar k}$ we conclude that $(\bar H)^l + \sum_{i=1}^l (g_i)_{i\bar\delta(h^e)}(\bar H)^{l-i} = 0$. For each $i$, since $\delta(g_i) \leq i\bar\delta(h^e)$, it follows that $(g_i)_{i\bar\delta(h^e)} \in \adelta$. Thus $\bar H$ is integral over $\adelta$, which completes the proof of the claim.
\end{proof}

Let $\tilde \delta := e\bar\delta$ and let $\tilde \scrF := \{\tilde F_d\}_{d \in \zz}$ be the corresponding filtration. Since $\tilde \delta \leq e\delta$, $\profinggg{e\delta} \subseteq \profinggg{\tilde\delta}$. But observe that $\tilde F_d = \{f : e\bar\delta(h) \leq d\} = \bar F_{\frac{d}{e}}$, and hence the mapping $t \to s^e$ gives an isomorphism $\profinggg{\bar\delta} := \bigoplus_{d \in \zz} \bar F_{\frac{d}{e}}t^{\frac{d}{e}} \overset{\cong}{\longrightarrow} \bigoplus_{d \in \zz}\tilde F_ds^d \cong \profinggg{\tilde\delta}$. Arguments similar
to those in the 2nd part of Lemma \ref{inf-intersection-preserve-lemma} allow to conclude that the restriction of the above map to $\profing$ gives a chain of inclusion $\profing \subseteq \profinggg{e\delta} \subseteq \profinggg{\tilde\delta}$. Once again as in the proof of lemma \ref{inf-intersection-preserve-lemma} we conclude $\profinggg{e\delta}$ is integral over $\adelta$, and by the above claim $\profinggg{\tilde\delta}$ is integral over $\adelta$. It follows that $\profinggg{\tilde\delta}$ is integral over $\profinggg{e\delta}$. Moreover, since $\adelta$ is a finitely generated $\kk$-algebra, it also follows that $\profinggg{\tilde\delta}$ is a finitely generated $\kk$-algebra.\\

By construction $\tilde \delta(f^m) = m\tilde \delta(f)$ for all $f$ and $m$. It follows that the ideal generated by $(1)_1$ in $A^{\tilde \delta}$ is radical. Therefore theorem \ref{gengfsgf-characterization} implies that $\tilde\delta$ is a quasidegree. This completes the proof of the first assertion of Theorem \oldref{thm-noeth-integral-quasidegree}. If $\delta$ is non-negative, then by construction $\tilde \delta$ is also non-negative, and applying lemma \ref{inf-intersection-preserve-lemma}, we deduce the second assertion.
\end{proof}

We summarize the results of theorems \ref{thm-noeth-integral-quasidegree}, \ref{filtrexistence-thm1} and \ref{filtrexistence-thm2} in the following

\begin{cor}\label{sgfiltrexistence}
Let $X$ be an affine variety of dimension $n$ and $A$ be the coordinate ring of $X$.
\begin{enumerate}
\item Let $V_1, \ldots, V_m$ be closed subvarieties of $X$ such that $\bigcap_i V_i$ is a finite set. Then there is a \complete \sgf $\delta$ on $A$ such that the corresponding completion $\psi_\delta$ of $X$ preserves the intersection of the $V_i$'s at $\infty$.

\item Let $f: X \to Y$ be a dominating map of affine varieties. Then there is a \complete \sgf $\delta$ on $A$ such that $\psi_\delta$ preserves $f$ at $\infty$. \qed
\end{enumerate}
\end{cor}
\subsection{A Quasifinite Map with Points at Infinity for any Semidegree} \label{subsec-semi-ctr-example}
\renewcommand{\filtrationchar}{\delta}

Let $X = Y = \affine[2]{\kk}$ and $f := ((x_1^2 - x_2^4)^2 + x_1x_2, (x_1^2 - x_2^4)^3 + x_1x_2):X \to Y$. Then $f$ is a quasifinite map. We show below that there is {\em no} complete semidegree $\delta$ on $A := \kk[x_1,x_2]$ such that $\psi_\delta$ preserves $f$ at $\infty$.\\

Let $a := (a_1, a_2) \in \affine[2]{\kk}$. As in section \oldref{subsec-filtrintro-existence}, let $H_{i}(a) := \{x \in X: f_i(x) = a_i\}$. Let $\delta$ be any complete \gf on $A$ with associated filtration $\scrF = \{F_d\}_{d \geq 0}$. Recall the notation of section \oldref{subsec-semiquasidegree-properties}: given an ideal $\ppp$ of $A$, let $\ppp^\delta := \dsum_{d \geq 0} (p \cap F_d)$ be the corresponding `homogenous' ideal of \profing. Let $\ppp_i = \langle f_i - a_i \rangle \subseteq A$. We will show that $V(\ppp^\delta_1, \ppp^\delta_2, (1)_1) \neq \emptyset$, or equivalently $\sqrt \scrI \subsetneqq \adelta_+$, where $\scrI := \langle \ppp^\delta_1, \ppp^\delta_2, (1)_1 \rangle$ and $\adelta_+$ is the maximal homogenous ideal of $\profing$.\\

Let $d_i := \delta(x_i)$. If $d_1>2d_2$, then $\delta(f_1-a_1)= \delta(x_1^4) = 4d_1$, and $\delta(f_1-a_1-x_1^4) < 4d_1$. Hence, $\langle \tilde\ppp_1, (1)_1 \rangle = \langle (f_1-a_1)_{4d_1}, (1)_1 \rangle = \langle ((x_1)_{d_1})^4, (1)_1 \rangle$. Similarly $\langle \tilde\ppp_2, (1)_1 \rangle = \langle ((x_1)_{d_1})^6, (1)_1 \rangle$. Thus $\scrI = \langle ((x_1)_{d_1})^4, (1)_1 \rangle$. But then $\sqrt \scrI \neq \adelta_+$, for if it were true, then $\scrI$ would be an $\adelta_+$-primary ideal in $\profing$ generated by $2$ elements. Since $\adelta_+$ is a maximal ideal of $\profing$, it would follow that dimension of $\spec \profing$ will be at most 2 (\cite{am}, Theorems 11.14 and 11.25). But this is impossible, since dimension of $\spec \profing$ is 3.\\

Similarly, if $d_1<2d_2$, then $\scrI = \langle ((x_2)_{d_2})^8, (1)_1 \rangle$, and by the same reasoning as above, $\sqrt \scrI \neq \adelta_+$. So assume $d_1 = 2d_2 = d$. Now $x_1^2 - x_2^4 = (x_1-x_2^2)(x_1+x_2^2)$. Note that $\delta(x_1 \pm x_2^2) \leq d$. But, since $\delta((x_1-x_2^2)+(x_1+x_2^2)) = d$, at least one of $x_1+x_2^2$ and $x_1-x_2^2$ has $\delta$-value $d$, whereas the other has $\delta$-value at least $1$ (since $\delta$ is complete). Thus $\delta((x_1-x_2^2)^2(x_1+x_2^2)^2) \geq 2d + 2 = 4d_2 + 2 > 3d_2 = \delta(x_1x_2)$. Hence $\delta(f_1-a_1 - (x_1^2-x_2^4)^2) < \delta(f_1 - a_1)$. It follows that $\langle \ppp^\delta_1, (1)_1 \rangle = \langle ((x_1^2-x_2^4)_{d'})^2,(1)_1 \rangle$, where $d' := \delta(x_1^2-x_2^4)$. Similarly, $\langle \ppp^\delta_2, (1)_1 \rangle = \langle ((x_1^2-x_2^4)_{d'})^3, (1)_1 \rangle$. Thus $\scrI = \langle ((x_1^2-x_2^4)_{d'})^2, (1)_1 \rangle$, and since $\scrI$ is generated by 2 elements, the same argument as in the previous paragraph shows that $\sqrt \scrI \neq \adelta_+$.\\

So we proved that for every complete \gf $\delta$ on $A$, $\psi_\delta$ does not preserve {\em any} fiber of $f$ at $\infty$. Hence none of the assertions of corollary \ref{sgfiltrexistence} would remain valid if we replace in its conclusion the `quasidegree' by a `semidegree'. \\

{\bf A quasidegree that preserves $f$ at infinity:} By corollary \ref{sgfiltrexistence}, we know that there exist completions of $x$ determined by complete \sgff s on $A$ which preserve $f$ at $\infty$. The filtration $\scrF$ corresponding to one such \sgf $\delta$ is defined as follows: $F_0 = \kk$, $F_1 = \kk\langle 1, x_2, x_1^2-x_2^4 \rangle$, $F_2 = (F_1)^2 + \kk\langle x_1 \rangle$, and $F_d = \sum_{j=1}^{d-1} F_jF_{d-j}$ for $d \geq 2$.\\

It is easy to see that $\delta = \max\{\delta_1, \delta_2\}$ where $\delta_1$ is the weighted degree on $A$ that assigns weight $1$ to $x_2$ and $-1$ to $x_1-x_2^2$, and $\delta_2$ is the weighted degree on $A$ that assigns weight $1$ to $x_2$ and $-1$ to $x_1+x_2^2$.

\begin{claim*}
$\psi_\delta$ preserves {\em every} fiber of $f$ at $\infty$.
\end{claim*}

\begin{proof}
Let $a := (a_1, a_2) \in \affine[2]{\kk}$. As above, let $\ppp_i$ be the ideal of $H_i(a)$ for $i = 1, 2$. Also define $\ppp^\delta_i$, $\scrI$, $\adelta_+$ as before. It suffices to show that for a homogenous prime ideal $P$ of $\profing$, if $P \supseteq \scrI$, then $P \supseteq \{(x_1)_2, (x_2)_1, (x_1^2-x_2^4)_1\}$ (since these latter elements generate the ideal $\adelta_+$).\\

Let $P \supseteq \scrI$ be a homogenous prime ideal of $\adelta$. Note that $(f_1-a_1)_3 = ((x_1^2 - x_2^4)_1)^2(1)_1 + (x_1)_2(x_2)_1 - a_1((1)_1)^3 \in \scrI$, which implies that $(x_1)_2(x_2)_1 \in \scrI$. Similarly $(f_2-a_2)_3 = ((x_1^2 - x_2^4)_1)^3 + (x_1)_2(x_2)_1 - a_2((1)_1)^3 \in \scrI$, which implies, via the previous inclusion, that $((x_1^2 - x_2^4)_1)^3 \in I \subseteq P$. But then $(x_1^2 - x_2^4)_1 \in P$. Also since $(x_1)_2 (x_2)_1 \in P$, it follows that either $(x_1)_2 \in P$ or $(x_2)_1 \in P$. But now note that
$$((x_1)_2)^2 - ((x_2)_1)^4 = (x_1^2 - x_2^4)_4 = (x_1^2 - x_2^4)_1 ((1)_1)^3.$$
Since the element in the right hand side lies in $P$, it follows that if either of $(x_1)_2$ and $(x_2)_1$ lies in $P$, then the other lies in $P$ as well! Hence we proved that $P \supseteq \{(x_1^2 - x_2^4)_1, (x_1)_2, (x_2)_1, (1)_1\}$.
\end{proof}

\section{Bezout Theorem} \label{sec-bezout}
\subsection{Bezout Theorem for Semidegrees} \label{subsec-bezout-semidegree}
Let $S = \bigoplus_{d \geq 0}S_i$ be any graded algebra over $\kk$. Recall that for any integer $d > 0$, the inclusion of the $d$-th truncated subring $S^{[d]} := \bigoplus_{i \geq 0} S_{id}$ into $S$ induces a map $\proj S \to \proj S^{[d]}$ which is an isomorphism of algebraic varieties. It is called the `$d$-uple embedding' of $\proj S$ \cite[exercise II.5.13]{hart}. Moreover, if $S$ is a finitely generated $\kk$-algebra, then with $d$ being equal to the number of generators of $S$ times a common multiple of degrees of all generators it follows that $S^{[d]}$ is generated by $S_d$. If $S_0 = \kk$, such a $d$-uple embedding gives a closed immersion of $\proj S$ into the projectivization of $S_d$.

\begin{thm}\label{thm-affine-bezout}
Let $X$ be an affine variety of dimension $n$ and $A$ be the coordinate ring of $X$. Let $\delta$ be a \complete semidegree on $A$. Let $(g_1)_{d_1}, \ldots, (g_N)_{d_N}$ be a set of generators of $\adelta$ and $d \geq 1$ be such that the $d$-uple embedding of $\xdelta$ is a closed immersion of $\xdelta$ into the projectivization $\pp_L$ of $F_d$.  Denote by $D$ the degree $\xdelta$ in $\pp_L$. Let $f = (f_1, \ldots, f_n):X \to \ank$ be any quasifinite map. Then for all $a \in \ank$, 
\begin{align}\label{semi-bezout}
|\finv(a)| \leq \frac{D}{d^n}\prod_{i=1}^n\delta(f_i), \tag{A}
\end{align}
where $|\finv(a)|$ is the size of the fiber $\finv(a)$ with all points counted with multiplicity. If in addition $\psi_\delta$ preserves $\finv(a)$ at $\infty$, then \ref{semi-bezout} holds with an equality.
\end{thm}

\begin{proof}
Let $a = (a_1, \ldots, a_n) \in \ank$. For each $i$, let $\ppp_i := \bigoplus_{j \geq 0} (\langle (f_i - a_i )^d\rangle \cap  F_j)$. Then the closure $\overline{V((f_i-a_i)^d)}$ of $V((f_i-a_i)^d)$ in $\xbar := \proj\adelta$ is precisely $V(\ppp_i)$. Since $\tilde Q_i := ((f_i)_{\delta(f_i)})^d \in \ppp_i$, it follows that $\overline{V((f_i-a_i)^d)} \subseteq V(\tilde Q_i)$. Now consider $\xbar$ as embedded into $\pp_L$ via $d$-uple embedding. Then $\tilde Q_i$ is restriction of a degree $\delta(f_i)$ homogenous $\hat Q_i$ in the homogenous coordinate ring of $\pp_L$. Thus $V(\tilde Q_i) = \xbar \cap V(\hat Q_i)$, where the `$V$' in the left hand side denotes zero set in $\xbar$ and the `$V$' in the right hand side denotes zero set in $\pp_L$. But then
\begin{align*}
V((f_1-a_1)^{d}) \cap \ldots \cap  V((f_n-a_n)^{d}) &\subseteq  V(\tilde Q_1) \cap \cdots \cap V(\tilde Q_n) \\
				&= \xbar \cap V(\hat Q_1) \cap \cdots \cap V(\hat Q_n) 
\end{align*}
Inequality \ref{semi-bezout} now follows from following pair of observations:

\begin{itemize}
\item The sum of the multiplicities of zero dimensional components of $\xbar \cap V(\hat Q_1) \cap \cdots \cap V(\hat Q_n)$ is at most $D\delta(f_1) \cdots \delta(f_n)$, and
\item $|V((f_1-a_1)^{d}) \cap \ldots \cap  V((f_n-a_n)^{d})| = d^n|V(f_1-a_1) \cap \ldots \cap  V(f_n-a_n)|$ when counted with multiplicity.
\end{itemize}

If $\delta$ is a semidegree, then $\ppp_i$ is generated by $\tilde Q_i$ and hence $V(\tilde Q_i) = V(\ppp_i) =  \overline{V((f_i-a_i)^d)}$. It follows that $\delta$ preserves $\finv(a)$ at $\infty$ iff
\begin{align*}
V((f_1-a_1)^{d}) \cap \ldots \cap  V((f_n-a_n)^{d}) &= \overline{V((f_1-a_1)^{d})} \cap \ldots \cap \overline{V((f_n-a_n)^{d})} \\
		&= V(\ppp_1) \cap \cdots \cap V(\ppp_n) \\
		&= V(\tilde Q_1) \cap \cdots \cap V(\tilde Q_n) \\
		&= \xbar \cap V(\hat Q_1) \cap \cdots \cap V(\hat Q_n) 
\end{align*}
This proves the theorem.
\end{proof}

\begin{rem}
 In fact, the conclusion of the theorem holds if $\delta = \max_{j=1}^N \delta_j$ is any quasidegree such that $\delta_1(f_i) = \cdots = \delta_N(f_i)$ for all $i=1, \ldots, n$.
\end{rem}

For the following proposition, assume $\kk = \cc$. A {\em valuation} on the field $\cc(X)$ of rational functions on $X$ with value group $\zz^n$ (equipped with an addition preserving total order) is a surjective map $\nu:\cc(X)\setminus\{0\} \to \zz^n$ such that:  
\begin{itemize}
\item $\nu(\lambda f) = \nu(f)$ for all $\lambda \neq 0 \in \cc$,
\item $\nu(f+g) \geq \max\{\nu(f),\nu(g)\}$ for all $f, g \neq 0 \in \cc(X)$,
\item for every pair of elements $f, g \neq 0 \in \cc(X)$ such that $\nu(f) = \nu(g)$, there is a $\lambda \neq 0 \in \cc$ such that $\nu(f-\lambda g) > \nu(f)$, and
\item $\nu(fg) = \nu(f) + \nu(g)$ for all $f, g \neq 0 \in \cc(X)$.
\end{itemize}
Fix a valuation $\nu$ on $\cc(X)$. Let $\delta$ be a \complete degree like function on $A$, and let $d$ and $D$ be as in Theorem \oldref{thm-affine-bezout}. Following \cite{khovanskii-kaveh}, we can associate a convex body $\Delta$ to $\delta$ such that $D$ is $n!$ times the $n$-dimensional volume $V_n(\Delta)$ of $\Delta$, namely: 

\begin{prop}
Let $C$ be the smallest closed cone in $\rr^{n+1}$ containing
$$G := \{(\frac{1}{d}\delta(f),\nu(f)) \in \zz_+^{n+1}: f \in A\}.$$ 
Let $\Delta$ be the convex hull of the cross-section of $C$ at the first coordinate value $1$. Then $D = n!V_n(\Delta)$.
\end{prop}

\begin{proof}
Let $L := F_d = \{f \in A: \delta(f) \leq d\}$. Let $\pp_L$ be the projectivization of $L$, and $\phi_L:\xdelta \into \pp_L$ be the $d$-uple embedding. Degree $D$ of $\phi_L(\xdelta)$ in $\pp_L$ is the number of common zeros of $n$ generic elements of $L$, and it is precisely the intersection index $[L,\ldots, L]$ of $n$ copies of $L$ as defined in \cite{khovanskii-kaveh}.  Since the mapping degree of $\phi_L$ is $1$, the proposition follows from the main theorem of \cite{khovanskii-kaveh}.   
\end{proof}

\begin{example}
Let $X = \cc^n$, and let $\nu$ be the `monomial valuation' on $A=\cc[x_1, \ldots, x_n]$, which assigns to $\sum a_\alpha x^\alpha \in A\setminus \{0\}$ the minimal (lexicographically) exponent $\alpha = (\alpha_1, \ldots, \alpha_n)$ among $\alpha$ with $a_\alpha \neq 0$. Let $\delta$ be a weighted degree on $A$ corresponding to weights $d_i$ for $x_i$, where $d_1, \ldots, d_n$ are positive integers. Then $(1)_1, (x_1)_{d_1}, \ldots, (x_n)_{d_n}$ is a set of generators of $\adelta$. Let $d$ be a sufficiently large multiple of $d_1, \ldots, d_n$. Then $\Delta = \{(1,x) \in \rr_+^{n+1}: \sum_{j=1}^n x_jd_j \leq d\}$ and $V_n(\Delta) = \frac{1}{n!}\prod_{j=1}^n\frac{d}{d_j}$. (Since $\psif$ preserves all fibers of the identity map $\mathbb{I}$ of $\ank$ at $\infty$, formula \ref{semi-bezout} for $\mathbb{I}$ implies $D = \frac{d^n}{d_1\cdots d_n}$ directly.) Hence, the right hand side of \ref{semi-bezout} is $\frac{\prod_j\delta(f_j)}{\prod_jd_j}$ and the result is the well known weighted version of Bezout's theorem (see e.g. \cite{damon}). 
\end{example}

Recall the construction of `iterated' semidegrees from example \ref{iterated-semi-example}. If $X = \ank$ and $\delta$ is a semidegree constructed from a weighted degree by repeating the `iteration' procedure finitely many times, we can do an explicit calculation of the number $D$ appearing in \ref{semi-bezout}.

\begin{thm}\label{iterated-degree-thm}
Let $X = \ank$ and let $\delta$ be a semidegree on $\kk[x_1, \ldots, x_n]$ constructed from a weighted degree $\delta_0$ by repeating the iteration procedure $k$ times. For each $i = 1, \ldots, k$, let $\delta_i$ be the semidegree obtained after $i$-th step, by fixing a polynomial $h_i$ which is prime with respect to $\delta_{i-1}$, and giving it a weight $w_i < \delta_{i-1}(h_i)$. Then $\frac{D}{d^n} = \frac{\delta_0(h_1) \cdots \delta_{k-1}(h_k)}{\delta_0(x_1)\cdots \delta_0(x_n)w_1\cdots w_k}$.
\end{thm}

\sloppy

\begin{proof}
From the proof of theorem \ref{iterated-thm}, it follows that $A^{\delta_i} = (A^{\delta_{i-1}}[s_i])^{\delta^e_{i-1}}/\langle (s_i - h_i)_{\delta_{i-1}(h_i)} \rangle$, where $s_i$ is an indeterminate and $\delta^e_{i-1}$ is induced by the extension $\delta_{i-1}^e$ of $\delta$ to $A^{\delta_{i-1}}[s_i]$, given by: $\delta_{i-1}^e(s_i) := w_i$. But then it follows by induction that for all $i \geq 1$,
$$A^{\delta_i} = \kk[x_0, \ldots, x_n, s_1, \ldots, s_i]/J_i,$$
where $J_i := \langle \tilde h_1 - x_0^{e_1-w_1}s_1, \ldots, \tilde h_i - x_0^{e_i-w_i}s_i \rangle,$ and for each $j$, $1 \leq j \leq i$, $e_j := \delta_{j-1}(h_j)$ and $\tilde h_j$ is a weighted homogenous polynomial in $x, s_1, \ldots, s_j$ which is a representative of $(h_j)_{e_j}$.  Hence $X^{\delta_k}$ is a complete intersection in $\pp^{n+k}(\kk;1,d_1,\ldots, d_n, w_1, \ldots, w_k)$, where $d_j := \delta_0(x_j)$ for $j = 1, \ldots, n$.  \\

Let $\tilde\delta$ be the weighted degree on $R_k := \kk[x_0, \ldots, x_n, s_1, \ldots, s_k]$ corresponding to weight $1$ for $x_0$, $d_j$ for $x_j$, $1\leq j \leq n$, and $w_j$ for $s_j$, $1 \leq j \leq k$. For $f_1, \ldots, f_n \in \kk[x_1, \ldots, x_n]$, let $\tilde f_1, \ldots, \tilde f_n \in R_k$ such that for each $i$, $\tilde f_i$ maps to $f_i$ under the above surjection and $\tilde\delta(\tilde f_i) = \delta(f_i)$. Now, for generic $f_1, \ldots, f_n$,
\begin{align*}
V(f_1, \ldots, f_n) &= \xbar \cap V(\tilde f_1) \cap \cdots \cap V(\tilde f_n) \\
			& =  V(\tilde h_1-x_0^{e_1-w_1}s_1) \cap \cdots \cap V(\tilde h_k - x_0^{e_k-w_k}s_k)  \cap V(\tilde f_1) \cap \cdots \cap V(\tilde f_n).
\end{align*}
By weighted homogenous Bezout theorem the right hand side equals $\frac{\delta(f_1) \cdots \delta(f_n)e_1\cdots e_k}{d_1\cdots d_n w_1\cdots w_k}$, and by the Bezout theorem for semidegrees, the left hand side is: $\frac{D}{d^n}\delta(f_1)\cdots \delta(f_n)$. 
\end{proof}

\fussy

\begin{example}
Let $\delta$ be the iterated semidegree on $\kk[x_1,x_2]$ from example \ref{iterated-semi-concrete-example}, so that $\delta(x_1) = 3$, $\delta(x_2) = 2$ and $\delta(x_1^2-x_2^3) = 1$. It follows by theorem \ref{iterated-degree-thm} that $D/d^n = \frac{6}{3\cdot 2 \cdot 1} = 1$.
\smallskip

It is easy to see that for any $k > 0$, $\psi_\delta$ preserves all fibers of the map $f_k := (x_1 + (x_1^2 - x_2^3)^2, (x_1^2 - x_2^3)^k):\affine[2]{\kk} \to \affine[2]{\kk}$. Hence, it follows from \ref{semi-bezout} that for any $k > 0$, the size of every fiber of $f_k$ is $3k$. Note that it is much smaller than the estimate $12k$ predicted by the weighted homogeneous version of Bezout's theorem.
\end{example}

\setcounter{section}{0}
\def\thesubsection{\thesection\arabic{subsection}}\ignorespaces

\def\thesection{Appendix}\ignorespaces
\section{}
\def\thesection{\Alph{section}}\ignorespaces
\begin{thm*}[Bernstein \cite{bern}]
For each $i = 1, \ldots, n$, let $A_i$ be a finite subset of $\zz^n$ and $f_i$ be a Laurent polynomial in $\cc[z_1, z_1^{-1}, \ldots, z_n,z_n^{-1}]$ such that $\supp(f_i) \subseteq A_i$. Then
\begin{align}\label{bernstein}
|V(f_1, \ldots, f_n)| \leq \scrM(A_1, \ldots, A_n) \tag{$*$}
\end{align}
where $|V(f_1, \ldots, f_n)|$ is the number of common isolated roots  of $f_1, \ldots, f_n$ in $(\cc^*)^n$ counted with multiplicity, and $ \scrM(A_1, \ldots, A_n)$ is the {\em mixed volume} of $A_1, \ldots, A_n$. Let $f_i = \sum_{\beta \in A_i} a_{i,\beta} x^\beta$, and for each $\alpha \in (\zz^n)^*$, let $A_{i,\alpha} := \{\beta \in A_i: \langle \alpha, \beta \rangle \leq \langle \alpha, \gamma \rangle\ \text{for all}\ \gamma \in A_i\}$ and $f_{i,\alpha} := \sum_{\beta \in A_{i,\alpha}} a_{i,\beta} x^\beta$. Inequality \mathref{bernstein} holds with an equality iff
\begin{align}\label{exactstein}
%\begin{tabular*}{0.75\textwidth}{p{0.75\textwidth}}
\text{for all }\ \alpha \in \zz^n,\ V(f_{1,\alpha}, \ldots, f_{n,\alpha}) = \emptyset. \tag{$**$}
%\end{tabular*}
\end{align}
\end{thm*}

Let $f_i$'s and $A_i$'s be as in Bernstein's theorem. Let $\scrP$ be the convex hull of  $A_1+ \cdots + A_n$. We only consider the case that $\dim\scrP =n$. Let $\phi_\scrP: \ntorus \into X_\scrP$ be the toric compactification of $\ntorus$ corresponding to $\scrP$.

\begin{claim*}
$f$ satisfies \mathref{exactstein} iff $\phi_\scrP$ preserves $f^{-1}(0)$ at $\infty$, where $f:=(f_1, \ldots, f_n): \ntorus \to \ntorus$.
\end{claim*}

\begin{proof}
Note that the condition \ref{exactstein} is not affected by
\begin{enumerate}[(i)]
\item translation of the $A_i$'s by elements in $\zz^n$, and
\item simultaneous transformation of $A_i$'s by an isomorphism of $\zz^n$.
\end{enumerate}
The above two transformations change the original $\scrP$ to a new polytope $\scrQ$ and $f_i$'s to new Laurent polynomial $g_i$'s. But $\scrP$ is isomorphic to $\scrQ$ via an isomorphism of $\zz^n$ (modulo a translation), which implies that $X_\scrP \cong X_\scrQ$ as toric varieties, i.e. in particular the $n$-torus in $X_\scrP$ corresponds to the $n$-torus in $X_\scrQ$. This implies that the property of preserving $f^{-1}(0)$ at $\infty$ is also invariant under the transformation $\scrP \to \scrQ$. Thus it will be safe to apply transformations of type (i) and (ii) in the proof.\\

At first assume $\phi_\scrP$ does not preserve $f^{-1}(0)$ at $\infty$. Let $\tilde x \in \overline{V(f_1)} \cap \cdots \cap \overline{V(f_n)} \cap X_\infty$. Recall that $X_\infty = \cup X_F$, where the union is over {\em faces} $F$ of $P := \scrP \cap \zz^n$ (i.e. $F = \scrF \cap \zz^n$, where $\scrF$ is a face of $\scrP$). Let $n\scrP \cap \zz^n := \{\alpha^0, \ldots, \alpha^N\}$. Then $X_\scrP$ is isomorphic to the closure of the image of $\phi_\scrP: \ntorus \to \pp^N$ which maps $z \mapsto [z^{\alpha^0}: \cdots: z^{\alpha^N}]$. Under this isomorphism, for a face $F := \scrF \cap \zz^n$ of $P$, $X_F$ is the closure of the image of the map $z \to [\phi_{\scrF,0}(z): \cdots: \phi_{\scrF,N}(z)]$, where 
\begin{align*}
\phi_{\scrF,j}(z) = \left\{
											\begin{array}{cl}
											z^{\alpha^j} & \text{if}\ \alpha^j \in n\scrF,\\
											0 & \text{otherwise}. 
											\end{array}
										\right.
\end{align*}
Let $F$ be the {\em smallest} face of $P$ such that $\tilde x \in X_F$. After an isomorphism of $\zz^n$, we may assume
\begin{enumerate}
\item $\alpha^0 = 0$ and $\alpha^i = e_i$ for $i = 1, \ldots, n$, where $e_i$ is the $i$-th unit vector of $\zz^n$,
\item $0 \in F \subseteq E_d := \zz\langle e_1, \ldots, e_d\rangle$, where $d = \dim(\conv F)$.
\end{enumerate}
Identify $X_\scrP$ with the closure of the image of $\phi_\scrP$ in $\pp^N$ as above. Let the homogenous coordinates of $\pp^N$ be $[x_0:\cdots:x_N]$. By assumption on $F$, $\tilde x_j \neq 0$ iff $\alpha_j \in \conv (nF)$, so in particular $\tilde x_0 \neq 0$. Fix an $i \in \{1, \ldots, n\}$. Since $\tilde x \in \overline{V(f_i)}$, there is a curve $C \subseteq V(f_i)$ such that $\tilde x \in \overline{C}$. Let the Puiseux expansion of $\overline{C}$ at $\tilde x$ be:
%$$x_j/x_0 = \tilde x_j/\tilde x_0 + \sum a_{j,k} t^{\gamma_{j,k}},$$
$$x_j/x_0 = \sum_{k \geq 0} b_{j,k} t^{\gamma_{j,k}},$$
where $\{\gamma_{j,k}\}_k$ is an increasing sequence of non-negative rational numbers. Since $\tilde x_j \neq 0$ for $j \leq d$, it follows that
\begin{align}
\gamma_{j,0} = 0\ \text{and}\ b_{j,0} = \tilde x_j/\tilde x_0\ \text{for}\ 1 \leq j \leq d. \label{project-b}
\end{align}

For all $t$ in a punctured disc centered at the origin in $\cc$, $C(t) \in \ntorus$. By our choice of coordinates, the preimage of $C(t)$ in $\ntorus$ is precisely $(x_1/x_0, \ldots, x_n/x_0)$. Let $\tilde \gamma := (\gamma_{1,0}, \ldots, \gamma_{n,0})$ and $\tilde b := (b_{1,0}, \ldots, b_{n,0}) \in \ntorus$. Then
$$\phi_\scrP(C(t)) = [1: {\tilde b}^{\alpha^1}t^{\langle \tilde \gamma, \alpha^1 \rangle} +\ \text{h.o.t.}:\  \cdots\ : {\tilde b}^{\alpha^N}t^{\langle \tilde \gamma, \alpha^N \rangle}+\  \text{h.o.t.}],$$
where {\em h.o.t.} abbreviates `higher order terms' (in $t$). Since $\lim_{t \to 0} \phi_\scrP(C(t)) = \tilde x$, and since $\tilde x_j \neq 0$ iff $\alpha^j \in \conv (nF)$, it follows that $\langle \tilde \gamma, \alpha^j \rangle > 0$ if $\alpha^j \not\in \conv(nF)$, and $\langle \tilde \gamma, \alpha^j \rangle = 0$ if $\alpha^j \in \conv(nF)$. But then $F$ is the `face' of $P$ corresponding to $\tilde\gamma$. This implies that $F = \conv(A_{1,\tilde\gamma} + \cdots + A_{n,\tilde\gamma}) \cap \zz^n$, where for each $j$, $A_{j,\tilde\gamma}$ is the `face' of $A_j$ corresponding to $\tilde \gamma$ as defined in the statement of Bernstein's theorem. Then,
\begin{align*}
f_i(C(t)) &= f_i(\sum_{k \geq 0} b_{1,k} t^{\gamma_{1,k}}, \ldots,  \sum_{k \geq 0} b_{n,k} t^{\gamma_{n,k}}) \\
	&= \sum_\beta a_{i,\beta}\prod_{j=1}^n (\sum_{k \geq 0} b_{j,k} t^{\gamma_{j,k}})^{\beta_j} \\
	&= \sum_\beta a_{i,\beta}({\tilde b}^\beta t^{\langle \tilde \gamma, \beta \rangle}+\ \text{h.o.t.}) \\
	&= t^{d_i}(\sum_{\beta \in A_{i, \tilde\gamma}} a_{i,\beta}{\tilde b}^\beta) +\ \text{h.o.t.},
\end{align*}
for some $d_i \in \qq_+$. Since $f_i(C(t))$ is identically zero, it follows that $f_{i, \tilde \gamma}(\tilde b) = \sum_{\beta \in A_{i, \tilde\gamma}} a_{i,\beta}{\tilde b}^\beta = 0$. Moreover, since $F \subseteq E_d$, there exists $\tilde\beta \in \zz^n$ such that $A_{i,\tilde\gamma} \subseteq \tilde\beta + E_d$. It follows that $\tilde f_i := x^{-\tilde\beta}f_{i,\tilde\gamma}$ depends only on $x_1, \ldots, x_d$, and for each $z \in \ntorus$, $f_{i, \tilde \gamma}(z) = z^{\tilde\beta}\tilde f_i(\pi(z))$, where $\pi(z)$ is the projection of $z$ to the first $d$ coordinates. Hence $f_{i, \tilde \gamma}(\tilde b) = {\tilde b}^{\tilde\beta} \tilde f_i(\pi (\tilde b)) = {\tilde b}^{\tilde\beta} \tilde f_i(\tilde x_1/\tilde x_0, \ldots, \tilde x_d/\tilde x_0)$ (by \ref{project-b}). It follows that $\tilde f_i(\tilde x_1/\tilde x_0, \ldots, \tilde x_d/\tilde x_0) = 0$.\\

Repeating the above arguments with every $i$, $1 \leq i \leq n$, we construct $\tilde\gamma^1, \ldots, \tilde\gamma^n \in \zz_+^n$, $\tilde\beta^1, \ldots, \tilde\beta^n \in \zz^n$ and $\tilde b^1, \ldots, \tilde b^n \in \ntorus$ such that for each $i$,
\begin{enumerate}[(i)]
\item $F = \conv(A_{1,\tilde\gamma^i} + \cdots + A_{n,\tilde\gamma^i}) \cap \zz^n$, so that $F$ is the `face' of $P$ corresponding to $\tilde\gamma^i$ and $A_{j,\tilde\gamma^i}$ is the `face' of $A_j$ corresponding to $\tilde\gamma^i$ for each $j$, 
\item $\pi(\tilde b^i) = (\tilde x_1/\tilde x_0, \ldots, \tilde x_d/\tilde x_0)$, and
\item $\tilde f_i := x^{-\tilde\beta^i}f_{i,\tilde\gamma^i}$ depends only on $x_1, \ldots, x_d$.
\end{enumerate}

%\sloppy

Since a decomposition of a face of a Minkowski sum of a polytope as a Minkowski sum of faces of summands is unique, it follows that for each $i,j,k$, $A_{i,\tilde\gamma^j} =  A_{i,\tilde\gamma^k}$. Let $\gamma := \tilde\gamma^1$ and $b := \tilde b^1$. Then for each $i$, $f_{i,\gamma}(b) = \sum_{\beta \in A_{i,\gamma}} a_{i,\beta}b^\beta = \sum_{\beta \in A_{i,\tilde\gamma^i}} a_{i,\beta}b^\beta
= f_{i, \tilde \gamma^i}(b) = b^{\tilde\beta^i}\tilde f_i(\pi(b)) = b^{\tilde\beta^i}\tilde f_i(\tilde x_1/\tilde x_0, \ldots, \tilde x_d/\tilde x_0) = 0$. Thus \ref{exactstein} is not satisfied!\\

%\fussy

Now assume \ref{exactstein} is not satisfied. Let $a := (a_1, \ldots, a_n) \in V(f_{1,\alpha}, \ldots, f_{n,\alpha})$. Fix an $i$, $1 \leq i \leq n$. We claim that there is a curve $C_i \subseteq V(f_i)$ with parametrization $C_i(t) = (a_1t^{\alpha_1} +\ \text{h.o.t.}\:, \ldots, a_nt^{\alpha_n} +\ \text{h.o.t.})$. To see it, let 
$$\hat f_i(t, w_1, \ldots, w_n) := t^{-d_i}f_i(t^{\alpha_1}(a_1+w_1), \ldots, t^{\alpha_n}(a_n + w_n)),$$
where $t, w_1, \ldots, w_n$ are indeterminates and $d_i := \langle \alpha, \beta \rangle$ for any $\beta \in A_{i, \alpha}$. Then $\hat f_i$ is a polynomial in $t, w_1, \ldots, w_n$, with {\em zero} constant term, and hence the origin $O$ belongs to the hypersurface $V(\hat f_i)$ of $\cc^{n+1}$. Since every point of a complex hypersurface (of a variety of dimension bigger than or equal to $2$) is the origin of a germ of a complex analytic curve contained in the hypersurface, it follows that for each $j$, $1 \leq j \leq n$, there is a convergent Puiseux series $w_j(t)$ in $t$ with positive exponents such that
$$\hat C_i(t) := (t, w_1(t), \ldots, w_n(t))$$
is the parametrization of a curve contained in $V(\hat f_i)$ and centered at $O$. But then $$C_i(t) :=  (t^{\alpha_1}(a_1+w_1(t)), \ldots, t^{\alpha_n}(a_n + w_n(t)))$$
satisfies the claim.\\

Then $\phi_\scrP(C_i(t)) = [a^{\alpha^0}t^{\langle \alpha, \alpha^0 \rangle} +$ h.o.t.$\: : \cdots\ :a^{\alpha^N}t^{\langle \alpha, \alpha^N \rangle} +$ h.o.t$]$. Multiplying each coordinate by $t^{-d}$ where $d := \min\{\langle \alpha, \alpha^j \rangle\}_{j=0}^N$, we see that $\lim_{t \to 0} \phi_\scrP(C_i(t)) = b := [b_0: \cdots: b_N]$, where
\begin{align*}
b_j &= \left\{
	\begin{array}{cl}
	a^{\alpha^j} 	& \text{if}\ \langle \alpha, \alpha_j\rangle = d \\
	0			&\text{otherwise}.
	\end{array}
	\right.
\end{align*}
Then $b \in \overline{C}_i \subseteq \overline{V(f_i)}$ for each $i$, and hence $\phi_\scrP$ does not preserve $f^{-1}(0)$ at $\infty$.
\end{proof}

\bibliographystyle{halpha}
\bibliography{C:/users/auniket/coxeter/kaaj/latex/utilities/bibi}

\end{document}